\documentclass{article}

\usepackage[english]{babel}

\usepackage[utf8]{inputenc}
\usepackage[utf8]{inputenc}
\usepackage{amsmath}
\usepackage{graphicx}
\usepackage{amssymb}
\usepackage{amsthm}
\usepackage{tikz-cd}
\usepackage{mathrsfs}
\usepackage[colorinlistoftodos]{todonotes}
\usepackage{enumitem}
\usepackage{yfonts}
\usepackage{ dsfont }
\usepackage{MnSymbol}
\usepackage{slashed}

\title{Survey on the metric SYZ conjecture and non-archimedean geometry}

\author{Yang Li}

\date{\today}
\newtheorem{thm}{Theorem}[section]
\newtheorem{lem}[thm]{Lemma}

\theoremstyle{definition}
\newtheorem{eg}[thm]{Example}

\newtheorem{conj}[thm]{Conjecture}

\newtheorem{rmk}{Remark}
\newtheorem{prop}[thm]{Proposition}
\newtheorem{Def}[thm]{Definition}
\newtheorem{Question}{Question}
\newtheorem*{Notation}{Notation}

\newtheorem*{Acknowledgement}{Acknowledgement}

\newcommand{\ie}{\emph{i.e.} }
\newcommand{\cf}{\emph{cf.} }

\newcommand{\R}{\mathbb{R}}
\newcommand{\C}{\mathbb{C}}

\newcommand{\N}{\mathbb{N}}
\newcommand{\Q}{\mathbb{Q}}

\newcommand{\norm}[1]{\left\lVert#1\right\rVert}
\newcommand{\Lap}{\Delta}

\DeclareMathOperator{\Tr}{Tr}

\begin{document}
	\maketitle

\begin{abstract}
We survey the metric aspects of the Strominger-Yau-Zaslow conjecture on the existence of special Lagrangian fibrations on Calabi-Yau manifolds near the large complex structure limit. We will discuss the diverse motivations for the conjectural picture, what the best hopes are, and a number of subtleties. The bulk of the survey highlights the role of pluripotential theory, and non-archimedean geometry in particular, with a list of open questions.
\end{abstract}

\section{Overview}

We survey the recent progress on the metric aspect of the Stominger-Yau-Zaslow conjecture, which concerns the existence of special Lagrangian fibrations on Calabi-Yau manifolds near the large complex structure limit. The paper is based substantially on \cite{LiFermat}\cite{LiNA}\cite{LiSkoda} and various subsequent talks.  A rough outline of the contents of chapter 2,3,4,5,6 is as follows:

\begin{itemize}
\item   We trace the diverse motivations of the SYZ conjecture from mirror symmetry, minimal surfaces, Riemannian geometry, complex and non-archimedean geometry. We discuss the most optimistic interpretation and its difficulties, before moving on to a more cautious weak version.

\item    We review the meaning of the large complex structure and essential skeleton from a complex geometric perspective, before a brief survey on the Kontsevich-Soibelman conjecture.

\item    We explain the various analytic ingredients. A brief overview of Yau's solution of the Calabi conjecture is given, to explain why new ideas are needed to understand the large complex structure limit.
Our primary focus is on complex pluripotential theory, which is the analytic core of \cite{LiFermat}\cite{LiNA}, alongside Savin's small perturbation theorem in elliptic PDE, and the regularity theory of real Monge-Amp\`ere equation.

\item  We include a minimalistic overview of non-archimedean pluripotential theory:  the notion of Berkovich spaces, semipositive metrices, and the non-archimedean Calabi metric. We also explain the intuition of the conjectural `comparison property', which is needed to give a differential geometric interpretation of the non-archimedean Calabi metric.

\item  We outline the proof of the recent progress \cite{LiFermat}\cite{LiNA}, emphasizing on the intuition, the subtleties, and the open problems. 

\end{itemize}

\begin{Acknowledgement}
The author is a current Clay Research Fellow and a MIT CLE Moore Instructor. He thanks L\'eonard Pille-Schneider for comments.
\end{Acknowledgement}

\section{Metric SYZ conjecture}

The Strominger-Yau-Zaslow conjecture is originally motivated by a combination of physical and differential geometric considerations, and stands at the crossroad of mirror symmetry, minimal surface theory, Riemannian geometry, complex K\"ahler geometry, and non-archimedean geometry. We will trace some historically significant developments, to see how 
the gradual realization of the analytic difficulties, has led to eclectic interpretations of the conjecture.

\subsection{The genesis of the SYZ conjecture}

An $n$-dimensional Calabi-Yau (CY) manifold $(X,\omega, \Omega)$ is a K\"ahler manifold with a nowhere vanishing holomorphic volume form $\Omega$, satisfying the complex Monge-Amp\`ere (MA) equation
\begin{equation}
\omega^n= \text{const} \Omega\wedge \overline{\Omega}, 
\end{equation}
which implies the Ricci flatness of the metric. Furthermore, such manifolds admit parallel spinors, hence are candidates for the target space metrics of supersymmetric type II string theories. Special Lagrangians of phase $\theta$ are $n$-dimensional submanifolds $L$ satisfying
\begin{equation}
\omega|_L=0,\quad \text{Im}(e^{-i\theta}\Omega)|_L=0.
\end{equation}
These are absolute minimizers within their homology classes, due to the calibration inequality
\begin{equation}
\int_L \text{Re}(e^{-i\theta}\Omega) \leq \int_L dvol_L =\text{Vol}(L)
\end{equation}
saturated precisely by the special Lagrangians. Physically, these correspond to the support of BPS D-branes.

The \textbf{Strominger-Yau-Zaslow conjecture} \cite{SYZ} in its primitive form asks:

\begin{conj}
Given a compact Calabi-Yau manifold $X$ near the large complex structure limit, can we find a \emph{special Lagrangian torus fibration} on $X$?
\end{conj}

The physical origin of the SYZ conjecture \cite{SYZ} comes largely from \textbf{mirror symmetry}, and a very brief sketch is as follows. From homological mirror symmetry, one expects a  compact Calabi-Yau manifold $X$ admits a mirror $X^\vee$, such that the category of D-branes on both sides are identified. On the holomorphic side $X^\vee$ (`B-side'), the points $x\in X^\vee$ support skyscrapper sheaves $\mathcal{O}_x$, which should correspond to certain Lagrangian branes inside the symplectic side $X$ (`A-side'). The extension groups $\text{Ext}^*(\mathcal{O}_x,\mathcal{O}_x)\simeq H^*(T^n)$, which suggests the Lagrangian branes are torus objects. For $x\neq y$, the Ext groups between $\mathcal{O}_x, \mathcal{O}_y$ would vanish, which suggests (inconclusively\footnote{The vanishing of Floer cohomology does not imply the vanishing of Floer cochain spaces, and there seems to be no strong argument to rule out intersecting special Lagrangians.}) that the tori are disjoint, leading to the speculation of the Lagrangian fibration structure. The assertion about \emph{special} Lagrangians, is however beyond mere homological mirror symmetry, and comes from the BPS condition on the D-branes. The moduli space of all $\mathcal{O}_x$ is the mirror manifold $X^\vee$, which should then be identified with the moduli space of BPS branes supported on the special Lagrangian tori. This moduli interpretation gives rise to a zeroth order approximation of K\"ahler structure on the mirror manifold $X^\vee$, subject to the higher order corrections related to the holomorphic discs (`instanton corrections'), whose effect is supposedly exponentially suppressed except near the singular fibres.  Ignoring the subtleties of singular fibres, then the SYZ picture offers a program to reconstruct the mirror, and interpret homological mirror symmetry as a version of Fourier-Mukai transform (`Mirror symmetry is T-duality').

\begin{Notation}
	Our convention is $d=\partial+ \bar{\partial}$,  $d^c= \frac{ \sqrt{-1} }{2\pi} (-\partial+ \bar{\partial})$, so $dd^c= \frac{\sqrt{-1}}{\pi}\partial \bar{\partial}$. The relation between K\"ahler potentials and K\"ahler metrics is $\omega_\phi= \omega+dd^c\phi$. Alternatively, we think of a K\"ahler metric in terms of local absolute potentials, meaning $\omega=dd^c\varphi$ for locally defined psh functions $\varphi$. Given a Hermitian metric $h$ on a line bundle $L$, its curvature form is $-dd^c\log h^{1/2}$ in the class $c_1(L)$.
\end{Notation}

Differential geometrically, the main evidence presented in the SYZ paper is the \textbf{semiflat metrics}. Consider the logarithm map $\text{Log}_t: (\C^*)^n\to \R^n$ over some open convex subset  $U\subset\R^n$,
\[
\text{Log}_t (z_1,\ldots z_n)=\frac{1}{\log |t|} (\log |z_1|,\ldots \log |z_n|). 
\]
Imposing $T^n$ symmetry, then by an elementary Hessian computation, $\phi$ is a smooth strictly convex function downstairs on $U$ if and only if its pullback to $\text{Log}_t^{-1}(U)$ is a smooth K\"ahler potential, and the Calabi-Yau condition 
\[
(dd^c\phi\circ \text{Log}_t)^n= \frac{\text{const}}{|\log |t||^{2n}} \prod d\log z_i\wedge d\overline{\log z_i}
\]
is equivalent to the \textbf{real Monge-Amp\`ere equation}
\[
\det (D^2 \phi) =\text{const}. 
\]
In this setting, the metric $dd^c\phi\circ \text{Log}_t$ is called semiflat, because its restriction to the $T^n$ fibres are Euclidean, due to $T^n$-symmetry. With respect to the Calabi-Yau structure
\[
\omega=dd^c (\phi\circ \text{Log}_t),\quad \Omega= \sqrt{-1}^n\prod d\log z_i,
\]
the $T^n$ fibres are \emph{special Lagrangians} of phase zero. The purpose of introducing the normalization parameter $t$, is that as $t\to 0$, the $T^n$-fibres shrink down to zero size, and the Calabi-Yau metrics $dd^c(\phi\circ \text{Log}_t)$ converge to the \textbf{real Monge-Amp\`ere metric} on $U\subset \R^n$
\begin{equation}\label{realMAmetric}
g_0= \frac{1}{2\pi}\sum_{i,j} \frac{\partial^2 \phi}{\partial x_i\partial x_j} dx_i dx_j, \quad \det(D^2\phi) =\text{const}. 
\end{equation}

Now near the \textbf{large complex structure limit}, which is a certain limiting situation for a family of Calabi-Yau metrics,
it is expected that the semiflat metrics emerge as an asymptotic description of the degenerating Calabi-Yau metrics, in the \emph{generic region} of the Calabi-Yau manifolds. In the SYZ picture, the generic region heuristically means away from the singular special Lagrangian fibres. The main point is that in the generic region, the special Lagrangians are just small perturbations of the logarithm maps in local toric charts. As we approach the large complex structure limit, the percentage of the Calabi-Yau volume measure occupied by the generic region should tend to $100\%$.

This recount of this SYZ heuristic reasoning underlines a few precautions:

\begin{itemize}
\item The metric predictions are more compelling in the \emph{generic region}. The non-generic region is subject to instanton  correction effects, whose metric significance 
is much more difficult to analyze.

\item The SYZ fibration is likely an \emph{emergent} behaviour near the large complex structure limit. In particular, the conjecture is  concerned with a \emph{family} of Calabi-Yau manifolds, and features such as semiflat metrics would only appear sufficiently close to the limit.

\end{itemize}

\subsection{Further motivations}

Aside from its importance in \textbf{mirror symmetry}, the SYZ conjecture is interesting for other diverse fields such as minimal surface theory, Riemannian geometry, K\"ahler and algebraic geometry.

\begin{itemize}
    
\item Special Lagrangians are \textbf{minimal submanifolds}, and in fact calibrated submanifolds. Currently, there are few methods for producing special Lagrangians in sufficiently large supply of Calabi-Yau manifolds, although there is a series of conjectures initiated by Thomas-Yau \cite{ThomasYau}\cite{Thomas} and further developed in \cite{Joyceconj}\cite{LiThomasYau}.

\item  The behaviour of a family of \textbf{Einstein metrics} depends strongly on whether the volume of geodesic balls satisfies the noncollapsing condition for some uniform constant $\kappa>0$.
\[
\text{Vol}_g(B_g(r)) \geq \kappa r^{\dim_\R X}, \quad \forall 0< r\leq \text{diam}(X).
\]
A good convergence and regularity theory is available in the non-collapsing case \cite{CheegerNaber}. On the other hand, metric degeneration in the collapsing case is largely terra incognita in Riemannian geometry, and the semiflat metric asymptote is a highly nontrivial emergent feature for a \textbf{collapsing} family of Calabi-Yau metrics.

\item A recurring theme of K\"ahler geometry is the interplay between metric  and complex geometry. 
 K\"aher-Einstein metrics in the non-collapsing case is tied to projective geometry \cite{DonaldsonSun}. The large complex structure limit is a very severe kind of polarized degeneration, whose transcendental behaviour (related to exponential and logarithms) is not adequately captured by traditional projective geometry, and instead  \textbf{non-archimedean} geometry stands out as a natural framework. One can then ask about the relation between Calabi-Yau metrics and non-archimedean geometry, a problem that turns out to be related to the SYZ conjecture.

\end{itemize}

\subsection{Collapsing K3 surfaces with elliptic surfaces}\label{GrossWilson}

An influential early work related to the SYZ conjecture is the gluing description of Gross-Wilson \cite{GrossWilson} concerning the hyperk\"ahler metric on K3 surfaces $X$ with elliptic fibrations $\pi:X\to \mathbb{CP}^1$.

	 By hyperk\"ahler rotation, the special Lagrangian torus fibres can be viewed as the holomorphic elliptic curves in a different complex structure. In the generic situation, the elliptic fibration has 24 $I_1$-type singular fibres, namely the local singularity in the fibration is modelled on $(z_1,z_2)\mapsto z_1z_2$ complex geometrically. 
 They fix K\"ahler classes $[\omega_X]$ on the K3 surface, and $[\omega_{\mathbb{P}^1}]$ on $\mathbb{CP}^1$, and describe the Calabi-Yau metrics $\omega_\tau$ in the K\"ahler class $\tau[\omega_X]+\pi^*[\omega_{\mathbb{P}^1}]$ for $0<\tau\ll 1$. Some key conceptual features are:

\begin{itemize}
\item In the generic region, the metrics $\omega_\tau$ are up to exponentially small errors modelled on \emph{semiflat metrics}, which can be explicitly described via the periods integrals of the elliptic curves. The subset of the K3 surface on which the semiflat metric asymptote breaks down, has length scale $O(\tau^{1/2}|\log \tau|^{1/2})$.

\item As $\tau\to 0$, the metrics $\omega_\tau$ converge in the Gromov-Hausdorff sense to a singular metric on $\mathbb{CP}^1$. The diameter of the K3 surfaces is of constant order $O(1)$. The length scale of generic elliptic curve fibres is $O(\tau^{1/2})$, which shrinks to zero size as $\tau\to 0$.

\item
In the neighbourhood of the $I_1$ type singular fibres, the metrics are modelled on the \emph{Ooguri-Vafa metrics}, which are explicit $S^1$-invariant incomplete hyperk\"ahler metrics constructed by means of the \emph{Gibbons-Hawking ansatz}.

\item The asymptotic geometry of the Ooguri-Vafa metric matches with the semiflat metric. This is a basic requirement for the gluing construction.

\item Near the nodal singular point of the $I_1$-fibre, the Ooguri-Vafa metric contains a small region approximated by the \emph{Taub-NUT metric}, whose length scale is $O(\tau^{1/2}|\log \tau|^{-1/2})$. These regions concentrate almost the entire $L^2$-Riemannian curvature of the K3 surface. 

\end{itemize}

The Gross-Wilson picture represents the \emph{best hope} on the SYZ conjecture,\footnote{As a caveat sometimes overlooked in the literature, the hyperk\"ahler rotation of the Gross-Wilson setting is not quite a polarized degeneration family, hence does not quite fit into our notion of large complex structure limit. In our perspective, Gross-Wilson is an inspiration, rather than an example of the SYZ conjecture.
	}  which includes an almost explicit description of the metric, and a special Lagrangian torus fibration exists globally on $X$. It is partially generalized to higher dimensional hyperk\"ahler manifolds with holomorphic Lagrangian abelian variety fibrations. After hyperk\"ahler rotation, these can be regarded as special Lagrangian fibrations. Tosatti et al. \cite{Tosatti} \cite{GrossTosattiZhang} established the semiflat metric asymptote in the generic region, and Gromov-Hausdorff collapse to the base manifold.

\subsection{Best hope on Calabi-Yau 3-folds}\label{BesthopeCY3}

For general information on this section, see  \cite[Chapter 8,9]{Joycebook}, and the introduction in \cite{LiTaub}.
In the initial years following the SYZ proposal, there was an overly optimistic belief based on the analogy with the hyperk\"ahler case, and based on topological and complex geometric considerations:

\begin{itemize}
\item The special Lagrangian fibration exists globally and is defined by a $C^\infty$ map, even though some fibres may be singular.

\item The discriminant locus on the base is codimension two. For Calabi-Yau 3-folds, under suitable genericity assumption, the discriminant locus is a trivalent graph, with two types of vertices, known as positive and negative vertices.\footnote{The names `positive/negative vertex' come from some old fashioned topological models of the torus fibration where the most singular fibres have Euler characteristics $\pm 1$ respectively. }

\item Along the edges of the trivalent graph, the singularity of the SYZ fibration is transversely modelled on the $I_1$ singularity.

\item The local region near the \emph{positive vertex} is complex geometrically a large open subset inside
\[
\{  z_0z_1z_2=1-z_3      \}\subset \C^3\times \C^*_{z_3}, \quad \Omega\propto \frac{1}{z_3}dz_0\wedge dz_1\wedge dz_2,
\]
while the \emph{negative vertex} region is modelled on a large open subset inside
\[
\{ z_3z_4=1-z_1-z_2 \} \subset  (\C^*)^2_{z_1,z_2}\times \C^2_{z_3,z_4} ,\quad \Omega\propto \frac{1}{z_1z_2} dz_2\wedge dz_3\wedge dz_4.
\]

\end{itemize}

The na\"ivete was challenged by Joyce \cite{Joyce}, based on his observations concerning \emph{special Lagrangian singularities}, which are markedly different from those visible in holomorphic fibrations. 
A global special Lagrangian fibration on the compact Calabi-Yau manifolds near the large complex structure, should it exist at all, is expected to have a much more subtle structure:

\begin{itemize}
	\item The special Lagrangian fibration is typically not defined by $C^\infty$ maps, but are at best piecewise smooth.

	\item The discriminant locus on the base is typically not codimension two, but the trivalent graph is expected to be thickened to a codimension one `ribbon'. The amount of thickening probably tends to zero in the large complex structure limit.

	\item The $I_1\times \R$ local singularity model does not lead to a Fredholm deformation theory for the singular special Lagrangian fibres, so should be replaced by some other singularity models.

\end{itemize}

Stepping aside from the substantial difficulties of the special Lagrangian local singularities, another major difficulty is to understand the Calabi-Yau metrics near the large complex structure limit, in complex dimension three. The optimistic expectations are inspired by the Gross-Wilson picture in complex dimension two. Hypothetically, the Calabi-Yau 3-fold is Gromov-Hausdorff close to a real 3-dimensional manifold, whose topology is believed to be the 3-sphere\footnote{This expectation comes from the topology of the essential skeleton, see section \ref{KontsevichSoibelmanconj} below.}, containing a trivialent graph, such that

\begin{itemize}
\item Away from the trivalent graph, the Calabi-Yau metric is semiflat up to exponentially small errors.

\item Transverse to the edges in the trivalent graph, the metric is modelled on the Ooguri-Vafa metric appearing in Gross and Wilson's picture.

\item  Near the positive and negative vertices of the trivalent graph, the local metric is modelled on some generalization of the Ooguri-Vafa metrics. 
\end{itemize}

It was recently realized that there exist almost canonical constructions of Ooguri-Vafa type metrics in complex dimension three, with the predicted topology and complex structure of the positive and negative vertices, constructed from a (nonlinear) generalized Gibbons-Hawking ansatz \cite{LiTaub}.  The asymptotic geometry of these Ooguri-Vafa type metrics matches with semiflat metrics in the generic region. The positive vertex metric contains a local region modelled on a generalized Taub-NUT type metric on $\C^3$, analogous to the way the Ooguri-Vafa metric contains a region modelled on the Taub-NUT metric.

The following questions are widely open:

\begin{Question}
Do such Ooguri-Vafa type metrics arise as blow up limits on any \emph{compact} Calabi-Yau manifolds near the large complex structure limit? 
\end{Question}

\begin{Question}
Can one give a gluing description of the CY metrics for 3-folds near the large complex structure, eg. in the case of quintic hypersurfaces?
\end{Question}

\begin{Question}
What kind of special Lagrangians can arise on $C^\infty$-small perturbations of these Ooguri-Vafa type metrics?
\end{Question}

\subsection{Strong vs. weak SYZ conjecture}

Due to the analytic difficulty of the SYZ conjecture, especially the nongeneric regions with large Riemannian curvature, the literature has developed many interpretations of the conjectures, with somewhat diverging goals. 

\begin{itemize}
    \item  (\textbf{Soft versions}) For applications to homological mirror symmetry, one is primarily interested in constructing Lagrangian fibrations without requiring $\text{Im}(\Omega)|_L=0$, and therefrom build a mirror manifold $X^\vee$ and prove the categorical predictions $D^bCoh(X^\vee)\simeq D^\pi Fuk(X,\omega)$. This viewpoint separates the symplectic and holomorphic data of the Calabi-Yau manifold, and has a topological/algebraic flavour.

    \begin{rmk}
    The \emph{special Lagrangian} condition usually left out of homological mirror symmetry discussions, is supposedly related to \emph{Bridgeland stability conditions}, which is a popular categorical interpretation of the BPS condition on D-branes.    
    \end{rmk}

    \item (\textbf{Strong metric version}) On compact Calabi-Yau manifolds near the large complex structure limit, find a global special Lagrangian torus fibration. 
    
    \item  (\textbf{Weak metric version}) Prove for a suitable class of compact Calabi-Yau manifolds near the large complex structure limit that a special Lagrangian $T^n$-fibration exists on a large subset with at least $99\%$ of the Calabi-Yau measure. More precisely, the percentage converges to $100\%$ in the limit.

\end{itemize}

The metric versions are highly sensitive to the Calabi-Yau metric, which is a much more rigid structure compared to the soft versions. They conform to the PDE spirit of the SYZ paper, while the soft versions are closer to the mirror symmetry motivations of the SYZ conjecture.

The strong metric version is perhaps the most faithful to the original intention of the SYZ paper. Its main evidence comes from the special case of hyperk\"ahler metrics, mentioned in section \ref{GrossWilson}. Other peripheral evidence comes from the construction of Lagrangian fibrations in many examples, ignoring the $\text{Im}(\Omega)|_L=0$ condition \cite{MatessiLag}. There are however many subtleties besetting this strong version, mentioned in section \ref{BesthopeCY3}, making the conjecture very formidable, and by comparison the supporting evidence seems  inadequate. Notably, the special Lagrangian singularities are not sufficiently understood, and we are not aware of any argument that definitively rules out special Lagrangians intersecting each other in the non-generic regions with large Riemannian curvature. As food for future thought, a somewhat weakened version bypassing these possible objections is

\begin{Question}
Given a compact Calabi-Yau manifold $X$ sufficiently near the large complex structure limit, is there an $n$-parameter family of special Lagrangian currents, whose supports sweep out all points on $X$?
\end{Question}

    
    
    

The weak metric version, on the other hand, concerns only the generic region, which is more accessible than the strong version. It conforms to the more cautious expectation, that the special Lagrangian fibration is only a limiting phenomenon. The bulk of the survey will focus on this weak version.

\section{Large complex structure limit}\label{Largecomplexstructurelimit}

The SYZ paper does not make precise the notion of large complex structure limit, and several non-equivalent interpretations are available in the current literature. We shall place the large complex structure limit in the framework of \textbf{polarized degenerations}. Intuitively, a polarized degeneration is when we fix the symplectic structure inside an integral class, and vary the complex structure in an algebraic one-parameter family so that it becomes singular in the limit. The large complex structure limit is the additional requirement that the degeneration is `as severe as possible'.

We work over $\C$. To set the scene,

\begin{itemize}
	\item 
	Let $S$ be a smooth affine algebraic curve, with a point $0\in S$.
	An \emph{algebraic degeneration family} is given by a submersive projective morphism $\pi: X\to S\setminus \{0\}$ with smooth connected $n$-dimensional fibres $X_t$ for $t\in S\setminus \{0\}$. This is in contrast with the \emph{formal} setting over the punctured formal disc $\text{Spec}(K)$ with $K=\C(\!(t)\!)$.
	An algebraic degeneration induces a formal degeneration by base change.
	
	\item  
	A \emph{polarisation} is given by an ample line bundle $L$ over $X$. This specifies the K\"ahler class, up to rescaling conventions.

	\item 
	We say $\pi$ is a degeneration family of \emph{Calabi-Yau} manifolds if there is a trivialising section $\Omega$ of the canonical bundle $K_X$. Over a small disc $\mathbb{D}_t$ around $0\in S$, this induces holomorphic volume forms $\Omega_t$ on $X_t$ via $\Omega= dt\wedge \Omega_t$. The \emph{normalised Calabi-Yau measure} on $X_t$ is the probability measure
	\begin{equation}\label{CalabiYaumeasureeqn}
	d\mu_t= \frac{   \Omega_t \wedge \overline{\Omega}_t }{   \int_{X_t} \Omega_t \wedge \overline{\Omega}_t   }.
	\end{equation}
	The Calabi-Yau metrics $\omega_{CY,t}$ on $X_t$ are the unique K\"ahler metrics in the class $\frac{1}{ |\log |t| |  }c_1(L)$ such that 
	\begin{equation}
	\frac{ \omega_{CY,t}^n }{  \int_{X_t} \omega_{CY,t}^n  }=d\mu_t.
	\end{equation}

	\item 
	We say $\pi: X\to S\setminus \{ 0\}$ is a \textbf{large complex structure limit} of Calabi-Yau manifolds if the essential skeleton has the maximal dimension $n$ (to be explained below), and the degeneration family admits a semistable snc model over $S$.

\end{itemize}

\begin{rmk}\label{sncmodel}
Filling in the central fibre at $0\in S$ would involve the choice of a \emph{model} of $\pi: X\to S$, namely a normal flat projective $S$-scheme $\mathcal{X}$ together with an isomorphism with $X$ over the punctured curve $S\setminus \{ 0\}$. It is called an \emph{snc model} if $\mathcal{X}$ is smooth, and
	the central fibre over $0\in S$  is a simple normal crossing divisor in $\mathcal{X}$, such that the intersections of divisors are irreducible or empty. If furthermore the central fibre is reduced, it is called a \emph{semistable snc model}. Models can be analogously defined over the formal disc. The existence of snc models is a consequence of Hironaka's resolution theorem. They are highly nonunique. By the semistable reduction theorem \cite[chapter 2]{ToroidalembeddingsI}, after finite base change to another smooth algebraic curve $S'$, we can always find some semistable snc model for the degeneration family $X\times_S (S'\setminus \{0\})$, so the existence of a semistable snc model is not a substantial assumption. Everything here is quasi-projective. The choice of a model is very useful, but not intrinsic to the degenerating CY metrics.
\end{rmk}

\begin{eg}
A typical example of large complex structure limit is the Fermat family of Calabi-Yau hypersurfaces
\[
X_t=\{  Z_0Z_1\ldots Z_{n+1}+ t \sum_0^{n+1} Z_i^{n+2}=0          \}\subset \mathbb{CP}^{n+1}.
\]
As $t\to 0$, the algebraic limit is the union of $n+2$ projective planes. Intuitively, the central fibre is highly reducible, and the degeneration is very severe.
\end{eg}

\begin{eg}
Consider a family of quartic K3 surfaces degenerating to a nodal K3 surface. This is a typical example of a polarized degeneration which is \emph{not} a large complex structure limit. In fact, the central fibre is irreducible, and the nodal singularity is mild (Kawamata log terminal in the birational geometry terminology).
\end{eg}

\subsection{Volume asymptote and essential skeleton}\label{volumeasymptoteessentialskeleton}

Consider an algebraic Calabi-Yau degeneration family $X\to S\setminus \{0 \}$ as above. We  follow \cite{Boucksom1} to consider the asymptote of $\int_{X_t} \Omega_t \wedge \overline{\Omega}_t$ as $t\to 0$. Along the way, we will introduce the concept of dual intersection complexes and essential skeletons, which are simplicial complexes encoding the intersection patterns of divisors on the central fibre. An important lesson is that the measure theoretic limit of the Calabi-Yau manifolds is closer to simplicial complexes than algebraic varieties, indicating that the metric limit must be significantly different from Fubini-Study metrics associated to projective embeddings of bounded degree.


A very useful tool is to fill in the central fibre by choosing an \emph{snc model} (\cf Remark \ref{sncmodel}) $\mathcal{X}$ over $S$. The central fibre $\mathcal{X}_0$ is an snc divisor with components $E_i$ for $i\in I$, and we write $\mathcal{X}_0= \sum_{i\in I} b_i E_i$. In the special case of \emph{semistable snc models} $b_i=1$ for $i\in I$; this can always be achieved after finite base change. The canonical divisor $K_{\mathcal{X}}$ is supported on $\mathcal{X}_0$ as $K_X$ has a trivialising section $\Omega$. We may write $K_{\mathcal{X} }=\sum_i (a_i+b_i-1 )E_i $, so that the \emph{relative log canonical divisor}
\[
K^{log}_{\mathcal{X}/S  }:= K_{\mathcal{X}}- K_S + \mathcal{X}_{0,red}- \mathcal{X}_0=    \sum a_i E_i.
\]
Shifting all $a_i$ by a constant $\kappa$ is equivalent to multiplying $\Omega$ by $t^\kappa$, which gives an elementary factor $|t|^{2\kappa }$ to $\int_{X_t} \Omega_t \wedge \overline{\Omega}_t$. Thus we shall always assume $\min a_i=0$.

It is useful to introduce a \emph{quantitative stratification} on $X_t$ according to the intersection pattern of $E_i$. Let $E_J=\cap_{i\in J} E_i$ for $J\subset I$, which is irreducible if nonempty. Using the distance function of a fixed smooth background K\"ahler metric on $\mathcal{X}$, we can write 
\[
E_J^0=\{  q\in X_t| d(q, E_J) \ll 1            \} \setminus \{   q\in X_t| d(q, E_{J'}) \ll 1   , \quad \text{some } J'\supsetneq J           \}.
\]
Around $\emptyset \neq E_J\subset \mathcal{X}$, we denote $p=|J|-1$,
and introduce local coordinates $z_0, \ldots z_n$ on $\mathcal{X}$, such that $z_0, z_1, \ldots, z_p$ are the defining equations of $E_i$ for $i\in J$. The conditions on the divisors mean that away from deeper strata we may arrange $t= \prod_0^p z_i^{b_i}$, and
\[
\Omega= u_J  \prod_0^p z_i^{a_i+b_i} d\log z_i \wedge \prod_{p+1}^n dz_j
\]
for some local nowhere vanishing holomorphic function $u_J$. By definition $\Omega=dt \wedge \Omega_t$ along $X_t$, so on $E_J^0$
\[
\Omega_t = b_0^{-1} u_J z_0^{a_0}\ldots z_p^{a_p} \prod_1^p d\log z_i \wedge \prod_{p+1}^n dz_j,
\]
\[
\sqrt{-1}^{n^2} \Omega_t\wedge \overline{\Omega}_t = |b_0|^{-2} |u_J|^2 |z_0|^{2a_0 }\ldots |z_p|^{2a_p} \prod_1^p \sqrt{-1} d\log z_i \wedge d\log \bar{z}_i \wedge \prod_{p+1}^n \sqrt{-1} dz_j \wedge d\bar{z}_j.
\]
Notice also that the local equation $t= \prod_0^p z_i^{b_i}$ has $b_J= \gcd_{i\in J} b_i$ sheets of solutions. Using the polar coordinates by $z_i= e^{ x_i\log |t|  + \sqrt{-1}\theta_i }$ for $i\in J$, ones sees that the magnitude of $\int_{E_J^0} \sqrt{-1}^{n^2} \Omega_t\wedge \overline{\Omega}_t$ is  $O(|\log |t| |^l  )$ for $l= |\{ j\in J: a_j=0  \} |-1$.

The local logarithmic variables $x_i=\frac{ \log |z_i|}{ \log |t| }$ lie on the simplex
\[
\Delta_J= \{  \sum_0^p b_i x_i =1, \quad 0\leq x_i \leq 1 \}.
\]
These depend on the choice of $z_i$, but since the local defining equation of divisors differ by a nowhere vanishing holomorphic function, the ambiguity of $x_i$ is only $O( \frac{1}{ |\log |t|| } )$ for $0<|t|\ll 1$. Taking a more global viewpoint, the combinatorial pattern of how these simplices fit together exactly reflects the intersection pattern of the divisors $E_i$. Formally, this information is encoded in the \textbf{dual intersection complex} $\Delta_{\mathcal{X}}$ for the snc model $\mathcal{X}$: this is the polyhedral complex whose vertices $v_i$ correspond to $E_i$, and we assign a simplex $\Delta_J$ with vertices $v_i$ for $i\in J$ if and only if $E_J\neq 0$. The coodinates $x_j$ then define a \emph{piecewise integral affine structure} on $\Delta_{\mathcal{X}}$. Up to the above $O(\frac{1}{ |\log |t||  }  )$ ambiguity, we now have a \emph{logarithm map} $\text{Log}_{\mathcal{X}}: X_t\to \Delta_{\mathcal{X}}$, locally described by $x_i=\frac{ \log |z_i|}{ \log |t| }$. Consequently, the \textbf{`hybrid' space} $X  \sqcup \Delta_{\mathcal{X}}$ is equipped with a natural topology, so that a sequence of points $z_k\in X_t$ converges to $x\in \Delta_{\mathcal{X}}$ iff $t\to 0$ and $\text{Log}_{\mathcal{X} }(z_k)\to x$. 
The name `hybrid' refers to the mixture of algebraic varieties with simplicial objects, which is better suited for measure theoretic limits, than the algebraic family $\mathcal{X}$.

The measure also singles out a distinguished subcomplex $Sk(\mathcal{X})$, called the \textbf{essential skeleton}, consisting of the simplices in $\Delta_{\mathcal{X} }$ whose vertices correspond to $E_i$ with $a_i=0$. This is where the limit of the normalised CY measure is supported. The dimension of $Sk(\mathcal{X})$ is a measurement of how transcendental the degeneration $X$ is; it is reflected by the growth order of $\int_{X_t} \Omega_t\wedge \overline{\Omega}_t$. The largest possible value for the dimension is $n$.

In the case of a \textbf{large complex structure limit}, $\dim_\R Sk(\mathcal{X})=n$. Let us analyze the CY measure more explicitly, in a semistable snc model. For $E_J$ corresponding to an $n$-dimensional simplex in $Sk(\mathcal{X})$, on $E_J^0$
\begin{equation}
\sqrt{-1}^{n^2} \Omega_t\wedge \overline{\Omega}_t =  |u_J|^2  \prod_1^n \sqrt{-1} d\log z_i \wedge d\log \bar{z}_i.
\end{equation}
Here $u_J$ limits to its value $u_J(E_J)$ at the point stratum $E_J$, which is called the Poincar\'e residue of $\Omega$, and is easily seen to be independent of the choice of coordinates $z_i$. It is a consequence of the residue theorem on Riemann surfaces that $|u_J(E_J)|^2$ is independent of such $J$ \cite[Thm. 7.1]{Boucksom1}. Thus the pushforward to $\Delta_{\mathcal{X}}$ of the normalised CY measure (\ref{CalabiYaumeasureeqn})  converges smoothly in the interior of $\Delta_J$ to a constant multiple of the \emph{Lebesgue measure}:
\begin{equation}\label{measureconvergence}
\text{Log}_{ \mathcal{X} *  }d\mu_t= \text{Log}_{ \mathcal{X}*  }\frac{ \Omega_t\wedge \overline{\Omega}_t }{  \int_{X_t}  \Omega_t\wedge \overline{\Omega}_t    } \xrightarrow{t\to 0}    d\mu_0:=\text{Const}\cdot dx_1\ldots dx_n.
\end{equation}
Notice $dx_1\ldots dx_n$ is canonically defined due to the presence of an integral affine structure on $\Delta_J$. Viewed as a measure on $\Delta_{\mathcal{X}}$, the limit $d\mu_0$ has null measure on the complement of the $n$-dimensional faces of $Sk(\mathcal{X})$, as the integral of $d\mu_t$ in the corresponding region is $O(\frac{1}{ |\log |t|| } )$. The constant in (\ref{measureconvergence}) is independent of $J$ and its sole purpose is to make $d\mu_0$ a probability measure.

\subsection{The effect of blow up}\label{Blowup}

A major caveat is that snc models are highly non-unique, because one can always blow up a given snc model along some locus inside the central fibre. The intrinsic information concerning polarized degenerations, such as the limiting behaviour of metrics and volume measures, are independent of particular snc models.

The effect of blow up on the dual intersection complex and essential skeleton is well understood (\cf \cite[A.4]{KS}). The intuitive picture is as follows:

\begin{itemize}
\item If we blow up an snc model $\mathcal{X}$ along a smooth irreducible subvariety properly contained inside $E_J=\cap_{j\in J} E_j$, but not contained in any smaller intersection stratum, then the new dual intersection complex contains $\Delta_{\mathcal{X}}$, while introducing a new vertex, and some new wings over certain faces of $\Delta_{\mathcal{X}}$. The essential skeleton remains intact.

\item If we blow up an snc model $\mathcal{X}$ along some $E_J=\cap_{j\in J} E_j$, then the new dual intersection complex is a subdivision of $\Delta_{\mathcal{X}}$. If the simplex $\Delta_J$ corresponding to $E_J$ is a face of the essential skeleton $Sk(\mathcal{X})$, then there is an induced subdivision on $Sk(\mathcal{X})$, and otherwise the essential skeleton remains intact.   
\end{itemize}

While $\Delta_{\mathcal{X}}$ generally  becomes larger under blow ups, the essential skeleton $Sk(\mathcal{X})$ is only subdivided, and its piecewise affine structure (in particular its homeomorphism type) is a \textbf{birational invariant}, denoted as $Sk(X)$. The birational invariance is not surprising: the essential skeleton is the measure theoretic limit of $X_t$, a property independent of the choice of models.

\subsection{Kontsevich-Soibelman conjecture}\label{KontsevichSoibelmanconj}

In an attempt to extract limiting information from the SYZ picture, Kontsevich and Soibelman \cite{KS2}\cite{KS} proposed the following picture. Let $X_t$ be a polarized degeneration of Calabi-Yau manifolds near the large complex structure limit, then

\begin{itemize}
\item (K\"ahler class normalization) The rescaled metrics $\omega_{CY,t}\in \frac{1}{|\log |t||}c_1(\mathcal{O}(1))$ have nontrivial \emph{finite diameter} Gromov Hausdorff subsequential limits.

\item There is an \emph{affine structure} on the essential skeleton $Sk(X)$ away from a Hausdorff \emph{codimension two} singular subset. On the smooth locus, there is a Riemannian metric obtained as the Hessian of local solutions to the \emph{real Monge-Amp\`ere equation}. This metric agrees with the \emph{Gromov-Hausdorff limit}, which is conjecturally independent of the subsequence.

\item (Topology) Under the strict Calabi-Yau condition $h^{p,0}(X_t)=0$ for $p<n$, the essential skeleton is homeomorphic to $S^n$.

\end{itemize}

The heuristic idea is that the essential skeleton should be the base of the hypothetical SYZ fibration, and the SYZ fibration is approximated by logarithm maps, at least in the generic region. We briefly comment on the status of the Kontsevich-Soibelman conjecture (\cf also \cite[section 4.5]{LiNA}):

\begin{itemize}
\item The uniform diameter estimate independent of small $t$ \[
C^{-1}\leq \text{diam}(X,\omega_{CY,t})\leq C,
\]
is recently established in joint work with Tosatti \cite{LiTosatti}, using primarily Riemannian geometric methods. This together with Gromov compactness proves the K\"ahler class normalization prediction.

\item 
The prediction about the existence of a metrically preferred affine structure away from \emph{codimension two} on the base, is part of the general lore of the SYZ conjecture (\cf section \ref{BesthopeCY3}), even though there is insufficient evidence. Generally speaking, the simplicial complex structure on $Sk(X)$ induces a piecewise affine structure, and improving it to an affine structure away from codimension two, would require highly nontrivial choices. The author is not aware of a completely satisfactory answer to the following elementary question:

\begin{Question}
Given a one-parameter family of quartic K3 surfaces near the large complex structure limit, without special symmetry, how can we determine the location of singular points on $Sk(X)$? How can we write down the affine structure on the regular locus explicitly?
\end{Question}

\item 

The topology of the essential skeleton is an active research topic in birational geometry. 
The homeomorphism between $Sk(X)$ with $S^n$ is verified for many examples. In general, 
it is known \cite{NicaiseXu}\cite{NicaiseXuYu} that $Sk(X)$ is a `pseudomanifold', its rational homology groups agree with $S^n$, and its fundamental group has trivial profinite completion, but the actual homeomorphism type is still elusive, and supposedly requires appealing to the Poincar\'e conjecture.
\end{itemize}

\begin{rmk}
Gromov-Hausdorff convergence is a standard way to make sense of weak limits, but alternative weak notions are possible. Kontsevich and Soibelman \cite{KS}\cite{KS2} aim to establish non-archimedean geometry as a suitable framework for studying limits of Calabi-Yau manifolds and the consequences for mirror symmetry. Kontsevich and Tscinkel \cite{KontsevichTschinkel} initiated the attempt to build up non-archimedean pluripotential theory by imitating K\"ahler geometry, a task taken much further by Boucksom et al. \cite{Boucksom}\cite{Boucksom1}\cite{Boucksomsemipositive}\cite{Boucksomsurvey} (\cf section \ref{NAsurvey}). Kontsevich and Soibelman may have anticipated long before any rigorous definitions, that the Calabi-Yau metrics should converge in some \emph{potential theoretic} sense to a non-archimedean object, and the limiting information should be read off purely in terms of data on the essential skeleton.

\end{rmk}

\section{Analytical foundations}

\subsection{Yau's solution to the Calabi conjecture}

A cornerstone of K\"ahler geometry is Yau's celebrated proof of the Calabi conjecture, which implies

\begin{thm}\cite{Yau}\label{Calabiconjecture}
	A compact K\"ahler manifold $X$ with a nowhere vanishing holomorphic volume form $\Omega$ admits a unique Calabi-Yau metric $(g,\omega, J,\Omega)$ within any given K\"ahler class. 
\end{thm}

We now give a very sketchy account of (the modern view on) Yau's proof of Theorem \ref{Calabiconjecture}. Fix a background K\"ahler metric $\omega$ in the given class, and the task is to find a K\"ahler potential $\phi$ solving the \textbf{complex Monge-Amp\`ere equation} for any given density function $e^f$ satisfying the cohomological constraint $\int_X e^f\omega^n=\int_X\omega^n$,
\begin{equation}\label{Calabiconjecture1}
\omega_\phi^n=(\omega+ dd^c \phi)^n = e^f \omega^n, \quad \int_X \phi \omega^n=0,
\end{equation}
In particular for $e^f\omega^n= \text{const}\cdot \Omega\wedge \overline{\Omega}$ one obtains the Calabi-Yau metric. The uniqueness follows from a simple integration by parts argument.

The existence proof follows the continuity method, namely to deform through the space of K\"ahler metrics as $f$ deforms from zero to the desired function. Viewing (\ref{Calabiconjecture1}) as defining a map from the potential to the density function, we need to show the map is \textbf{submersive} and \textbf{proper}. The submersion property is a consequence of standard Hodge theory on the Laplacian. To show properness one needs \textbf{a priori estimates} on $\phi$ depending only on $X$ and bounds on $f$, and for that matter we are free to impose any order of regularity on $f$.

The first step is to bound $\norm{\phi}_{C^0}$ (known as `$C^0$-estimate'), using a technique called \textbf{Moser iteration}. We rewrite the equation as
\[
(1-e^f) \omega^n= -dd^c \phi \wedge (\omega^{n-1}+\omega^{n-1}\wedge \omega_\phi+\ldots+ \omega_\phi^{n-1}      ).
\]
For any $p>1$, multiply the equation by $\phi|\phi|^{p-2}$ and integrate by parts,
\begin{equation*}
\int_X |\nabla |\phi|^{p/2} |_{\omega}^2 \omega^n \leq \frac{C(n, \norm{f}_{L^\infty}  )p^2}{p-1} \int_X   |\phi|^{p-1} \omega^n.
\end{equation*}
The key is that a weaker norm on RHS controls a stronger norm on LHS. This is reverse to the direction of the Poincar\'e inequality 
\begin{equation*}
\int_X u^2 \omega^n \leq C\int_X |du|_\omega^2 \omega^n, \quad \int_X u=0,
\end{equation*}
and the \textbf{Sobolev inequality}
\begin{equation*}
\int_X |u|^{ \frac{2n}{n-1} } \omega^n \leq C (\int_X |du|^2_\omega \omega^n + \int_X |u|^2 \omega^n ).
\end{equation*}
The combined force is that $\norm{\phi}_{L^2_1} \leq C$, and one can inductively bound $\norm{\phi}_{L^p}$ for large $p$, which turns out to be uniform in $p$. Taking the limit $p\to \infty$, this gives a bound $\norm{\phi}_{L^\infty}\leq C(X, \norm{f}_{L^\infty})$.

The second step is to bound $\norm{ dd^c\phi}_{C^0}$ (known as $C^{1,\bar{1}}$-estimate), \ie to prove a uniform equivalence between $\omega$ and $\omega_\phi$. The starting point is 
a differential inequality by local calculations (in the analyst's convention of Laplacian)
\[
\Lap_{\omega_\phi} \log \Tr_\omega \omega_\phi \geq -C \Tr_{\omega_\phi} \omega- C .
\]
Subtracting a large enough multiply of the identity
\[
\Lap_{\omega_\phi} \phi=n- \Tr_{\omega_\phi} \omega,
\]
we get a differential inequality
\[
\Lap_{\omega_\phi}( \log \Tr_\omega \omega_\phi - C\phi)   \geq  \Tr_{\omega_\phi} \omega- C' .
\]
Using the a priori bound on $\phi$, 
an application of \textbf{maximum principle} then shows $\Tr_\omega \omega_\phi  \leq C$.

\begin{rmk}
	This $C^{1,\bar{1} }$-estimate argument has \textbf{global nature}: if we only know (\ref{Calabiconjecture1}) on a standard unit ball in $\C^n$ with a given $C^0$-bound on $\phi$, there are counterexamples for the $C^{1,\bar{1} }$-bound. The solution can develop singularities.	
\end{rmk}

The third step is to get higher order estimates. Standard elliptic theory implies it is sufficient to have a $C^{2,\alpha}$ bound on $\phi$. Evans-Krylov theory bridged the gap between $C^{1,\bar{1} }$-bound and $C^{2,\alpha}$-bound (\cf \cite{Siu} Chapter 2, Section 4 for details). This argument is of \textbf{local nature}, and crucially uses that $\log \det(\partial_i\partial_{\bar{j}} \phi  )$ is a \emph{concave} function of the matrix $(\partial_i\partial_{\bar{j}} \phi)$, and the main tool is a Harnack inequality. The geometric insight is that higher order regularity is a manifestation of the local Euclidean nature of K\"ahler manifolds.

Yau's proof strategy is highly influential and permeates the vast majority of works on CY metrics. The brief sketch above, however, highlights two reasons why it is difficult to adapt to the setting of SYZ conjecture:

\begin{itemize}
\item The CY metrics undergoing large complex structure limit are highly degenerate, to the extent that the Sobolev constant becomes too big, so that the Moser iteration technique does not give useful $C^0$-estimate on the potential.

\item The $C^{1,\bar{1}}$ bound has global nature, so in order to obtain useful estimates on the Calabi-Yau metric only in the generic region, the design of the maximum principle must hold globally. This is hard on highly degenerate manifolds, where Riemannian curvature cannot be globally uniformly bounded.

\end{itemize}

For these reasons, our approach to the SYZ conjecture has significant departure from Yau's proof. The potential estimates use \textbf{complex pluripotential theory} instead, and 
the metric estimate in the generic region uses a deep theorem of Savin in \textbf{elliptic PDE theory}, bypassing Yau's estimates.

\subsection{Complex pluripotential theory}

Complex pluripotential theory concerns the study of weak notions of K\"ahler potentials, and their wider implications on algebraic geometry and PDEs. They can be viewed as the generalization of potential theory on Riemann surfaces to several complex variables. Some common themes include:

\begin{itemize}
\item The weak compactness theory for the space of potentials. This is suited for variational methods in K\"ahler geometry.

\item The relations between potential theory and algebraic geometry, via Hodge theory, $\bar{\partial}$-operator, intersection theory, etc. This provides transcendental methods to birational geometry.

\item The weak formulation of the complex Monge-Amp\`ere equation, and a priori potential estimates under weak assumptions on the volume density, or in the presence of singularities for the ambient complex variety. This is useful for studying singularity formation of canonical metrics.

\end{itemize}

Unlike Yau's proof of the Calabi conjecture, whose techniques are typical of elliptic PDEs, complex pluripotential theory is more akin to complex analysis, and frequently provides stronger results. To build up the intuition, we will first review the standard versions in the literature, before stating our uniform versions, which are foundational to our approach to the SYZ conjecture.

\subsubsection{Skoda inequality}

An upper semicontinuous $L^1$-function $\phi$ on a coordinate ball in $\C^n$ is called plurisubharmonic (psh)  if
it satisfies the sub mean value inequality when restricted to complex lines; this implies
$dd^c \phi\geq 0$. The basic intuition is that regularity properties for psh functions in general dimensions  are analogous to subharmonic functions on Riemann surfaces. This is captured by the local version of the \emph{Skoda inequality}:

\begin{thm}\label{Skodabasicversion}
	(\cf \cite[Thm 3.1]{Zeriahi}) If $\phi$ is psh on $B_2\subset \C^n$, with $\int_{B_2} |\phi| \omega_E^n \leq 1$ with respect to the standard Euclidean metric $\omega_E$, then
	there are dimensional constants $\alpha$, $C$, such that 
	\[
	\log \int_{B_1} e^{-\alpha \phi} \omega_E^n \leq  C.
	\]

\end{thm}

\begin{rmk}
The Skoda inequality might be contrasted with subharmonic functions on the unit ball in $\R^{2n}$ for $n>1$, for which exponential integrability is far too much to expect.	
\end{rmk}

On a compact K\"ahler manifold $(Y, \omega)$, we say an upper semicontinuous $L^1$-function $\phi\in PSH(Y,\omega)$  if its sum with the local potential of $\omega$ is psh, so that
$\omega_\phi=\omega+ dd^c\phi\geq 0$. This is the generalised notion of K\"ahler potentials. The standard global analogue of the Skoda inequality is:

\begin{thm}\label{Skodastandardversion}\cite{Tian}
	On a fixed $(X, \omega)$, there are positive constants $\alpha$, $C$ depending only on $X, \omega$, such that
	\[
	\int_X e^{-\alpha \phi} \omega^n \leq C, \quad \forall \phi \in PSH(X,\omega) \text{ with } \sup \phi=0.
	\]
\end{thm}

In our applications, we need to work with a polarized algebraic degeneration of Calabi-Yau manifolds  $\pi: X\to S\setminus \{0\}$ near the large complex structure limit, as in section \ref{Largecomplexstructurelimit}. Let $\omega_{FS}$ be a fixed Fubini-Study metric on $(X,c_1(L))$ induced by a projective embedding via the sections of a high power of $L$, and use $\omega_{FS,t}= \frac{1}{|\log |t||} \omega_{FS}|_{X_t}$ to define a family of background metrics on $X_t$ in the class  $\frac{1}{|\log |t||} c_1(L)$. The normalization factor $\frac{1}{|\log |t||} $ is to ensure two different choices of Fubini-Study reference metrics would differ by a potential with $C^0$ norm of order $O(1)$ independent of small $t$. Recall $d\mu_t$ is the normalized Calabi-Yau measure.

 We adapted the Skoda inequality to a uniform version \cite{LiSkoda}:

\begin{thm}(Uniform Skoda estimate)\label{UniformSkodathm}
There are uniform positive constants $\alpha, A$ independent of $t$ for $0<|t|\ll 1$, such that for the normalised Calabi-Yau measures $d\mu_t$,
\[
\int_{X_t} e^{-\alpha u}  d\mu_t \leq A, \quad \forall u\in PSH(X_t, \omega_{FS,t}) \text{ with } \sup_{X_t} u=0.
\]

\end{thm}

The proof involves covering $X_t$ by plenty of small regions which look like standard balls in $\C^n$. An elementary but somewhat tricky construction of test function allows one to estimate $L^1$ norms of the local potentials. One then applies the local version of Skoda inequality to each small region, and sum over all regions.

\subsection{Estimate on pluripotentials}\label{Toolsfrompsh}

A basic problem in pluripotential theory is to estimate K\"ahler potentials.  Kolodziej  pioneered a method to achieve the following effects. The basic versions of his theorems work on a fixed ambient K\"ahler manifold $(Y,\omega)$, and deal with K\"ahler potentials $\phi\in PSH(Y,\omega)$ normalized to $\sup_Y \phi=0$.  

\begin{itemize}
\item  (\textbf{`Potential estimate'}) If the volume density of $\omega_\phi^n$ has some integrability control such as an $L^p$-bound
\begin{equation}\label{Lpdensity}
\int_Y |\frac{\omega_\phi^n}{\omega^n}|^p \leq C, \quad p>1,
\end{equation}
then  $\phi$ has an $L^\infty$ bound depending only on $(X,\omega), n, p,C$ \cite{Kolodziej}\cite{DemaillyPali}\cite{EGZ}. In fact this can be improved to a $C^\alpha$ H\"older bound on $\phi$ \cite{KolodziejHolder}. (Notice $p=1$ would not suffice, as the $L^1$-bound $\int_Y \omega_\phi^n \leq \int_Y \omega^n$ is automatic).

\item (\textbf{`$L^1$-stability estimates'}) Suppose $\phi,\psi\in PSH(Y,\omega)$ are both subject to the $L^p$ volume density integrability control (\ref{Lpdensity}), and the normalization $\sup_Y\phi=\sup_Y\psi=0$. If 
\begin{enumerate}
	\item Either $\phi, \psi$ are close together in the $L^1$-sense $\int_Y |\psi-\phi|\omega^n\ll 1$, (`$L^1$-potential stability')
	\\
	\item Or the volume densities of $\omega_\phi^n$ and $\omega_\psi^n$ are close together in the $L^1$-sense, namely the total variation $\int_Y |\omega_\phi^n-\omega_\psi^n|\ll 1$, (`$L^1$-volume stability')
\end{enumerate}
  Then $|\phi-\psi|$ is small in the $L^\infty$ sense with quantitative estimates \cite{KolodziejL1}.

\end{itemize}

\begin{rmk}\label{globalpositivityKolodziej}
To appreciate the strength of Kolodziej's results, these should be contrasted with the Poisson equation $\Lap u=f$ on a compact Riemannian manifold in dimensions at least three. If $f\in L^p$ for some $p>1$, then we can only deduce $u\in W^{2,p}$, and $W^{2,p}$ fails to embed into $L^\infty$ for $p$ close to one, so we cannot expect any a priori $L^\infty$ bound on $u$. Ultimately, the \emph{global positivity condition} $\phi\in PSH(X,\omega)$ makes the difference. 

\end{rmk}

\begin{rmk}
Kolodziej's proofs depend on his pluripotential theoretic `capacity decay' argument. 
The author was informed by Freid Tong that a good part of Kolodziej's results have found new proofs \cite{Tong1}\cite{Tong2}, inspired by the recent breakthrough of Chen and Cheng \cite{ChenCheng} on the constant scalar curvature K\"ahler equation.
\end{rmk}

In our applications, we need to work with a family of K\"ahler manifolds, and the estimates need to be uniform  under very severe complex structure degenerations, and allowing the total volume to collapse to zero. Some subtleties are:

\begin{itemize}
\item    The Calabi-Yau volume measure is very different from the volume form of a Fubini-Study metric (\cf section \ref{volumeasymptoteessentialskeleton}), so the idea of volume density with respect to a Fubini-Study style ambient metric is no longer appropriate. The replacement of (\ref{Lpdensity}) turns out to be a Skoda type estimate (\ref{Skodaassumption}).

\item   Unlike $L^\infty$ bounds, the H\"older norm depends strongly on the choice of the ambient metric, which specifies a choice of distance function. We think it is highly non-obvious how to make a semi-explicit choice uniformly in the family, and therefore we do not attempt to generalize H\"older estimates.

\item A technical problem in our applications involving the comparison of two potentials, only one potential has volume density control. Thus unlike the $L^1$-stability estimates above, our version treats the two potentials asymmetrically.

\end{itemize}

Our analogue of the potential estimate is

\begin{thm}\label{pluripotentialthm1}  (\cf \cite[section 2.2]{LiFermat})	Let $(Y, \omega)$ be a compact K\"ahler manifold, and $\phi\in PSH(Y,\omega)\cap C^0$, such that $\omega_\phi^n$ is an absolutely continuous measure. Assume there are positive constants $\alpha, A$, such that the Skoda type estimate holds with respect to $\omega_\phi^n$:
	\begin{equation}\label{Skodaassumption}
	\int_Y e^{-\alpha u}  \frac{\omega_\phi^n }{ \text{Vol}(Y) } \leq A, \quad \forall u\in PSH(Y,\omega) \text{ with } \sup_Y u=0.
	\end{equation}
Then we have	\begin{itemize}
\item   If $\sup_Y \phi=0$, then $\norm{\phi}_{C^0} \leq C(n,\alpha,A)$.

\item   For fixed $n,\alpha,A$, there is number $B(n,\alpha,A)$, such that if $\frac{ \int_{ \phi\leq -t_0} \omega_\phi^n}{  \text{Vol}(Y)  }< (2B)^{-2n}$ for some $t_0$, then $\min \phi\geq -t_0- 4B (\frac{ \int_{ \phi\leq -t_0} \omega_\phi^n}{  \text{Vol}(Y)} )^{1/2n}$.  
			
	\end{itemize}
	
\end{thm}

The strength of this result is that it still applies when the complex/K\"ahler structures are highly degenerate, as it distills the dependence on $(Y,\omega)$ to only 3 constants $n,\alpha,A$. The relations to the more standard version above can be explained as follows:

\begin{itemize}
\item For a fixed ambient K\"ahler metric, by the H\"older inequality and the standard Skoda inequality Theorem \ref{Skodastandardversion}, under the density integrability assumption (\ref{Lpdensity}), then for $\frac{1}{p}+\frac{1}{p'}=1$, and some small enough $\beta>0$,
\[
\int_Y e^{-\beta u}  \frac{\omega_\phi^n }{ \text{Vol}(Y) } \leq \frac{1}{ \text{Vol}(Y) } \left(  \int_Y |\frac{\omega_\phi^n}{\omega^n}|^p  \omega^n   \right)^{1/p} \left( \int_Y e^{-p'\beta u}      \omega^n    \right)^{1/p'} \leq \text{const}.
\]
This means the Skoda type estimate (\ref{Skodaassumption}) is a weaker assumption than (\ref{Lpdensity}), even in the standard setting.

\item The first part of the conclusion recovers Kolodziej's potential estimate.

\item The second part of the conclusion is about comparing the two potentials $\phi$ versus $-t_0$. The condition $\frac{ \int_{ \phi\leq -t_0} \omega_\phi^n}{  \text{Vol}(Y)  }< (2B)^{-2n}$ means $\int_{ \phi\leq -t_0}\omega_\phi^n$ is small, which by the H\"older inequality and the density integrability (\ref{Lpdensity}) follows from the smallness of 
 $\int_{ \phi\leq -t_0}\omega^n$. This should be viewed as one half of the $L^1$-potential stability condition
$
 \int_Y |\phi+t_0|\omega^n\ll 1. 
$
 The conclusion for the lower bound on $\min \phi$, should be viewed as one half of 
 a smallness bound on the $C^0$-norm of $\phi+t_0$, namely the two potentials $\phi$ and $-t_0$ are close together in $C^0$-norm. 

\item 
To generalize to the cases with $\psi$ not necessarily constant, it suffices to replace $\omega$ by $\omega_\psi$, and $\phi$ by $\phi-\psi$ in Theorem \ref{pluripotentialthm1}.

\end{itemize}


Combining Thm. \ref{UniformSkodathm} with Thm. \ref{pluripotentialthm1}, we immediately obtain

\begin{thm}\label{UniformLinfty}\cite[Thm. 1.4]{LiSkoda}
	(Uniform $L^\infty$-estimate) Given a large complex structure limit of Calabi-Yau manifolds, the potential of the Calabi-Yau metrics $\omega_{CY,t}$ in the class $\frac{1}{|\log |t||}c_1(L)$ relative to a fixed Fubini-Study reference metrics $\omega_{FS,t}$, have \textbf{uniform $L^\infty$-estimate} independent of $0<|t|\ll 1$, under suitable additive normalization.
\end{thm}

Our adaption of the $L^1$-volume stability estimate is

\begin{thm}\label{UniformL1stabilitythm}(\cf \cite[Theorem 2.6]{LiNA})
	(Uniform $L^1$-stability) Let $(Y,\omega)$ be a compact K\"ahler manifold, and $\phi\in PSH(Y,\omega)\cap C^0$, satisfying the complex MA equations
	\[
	\frac{\omega^n }{ \text{Vol}(Y) }=d\mu,  \quad \frac{\omega_\phi^n }{ \text{Vol}(Y) }=d\nu
	\]
	for probability measures $d\mu$ and $d\nu$.
	Assume
	\begin{itemize}
		\item There is a Skoda type estimate
		\[
		\int_Y e^{-\alpha u}  d\mu \leq A, \quad \forall u\in PSH(Y,\omega) \text{ with } \sup_Y u=0.
		\]

		\item The complement of $E_0=\{ \phi>0  \}$ has a mass lower bound \[ \int_{  E_0^c} d\mu \geq \lambda>0. \]
		
		\item ($L^1$-stability assumption) The total variation $\int_Y |d\mu-d\nu| \leq s^{2n+3}<1$. 
		
		\item $\phi$ is smooth away from a (possibly empty) closed subset $S$ with $d\mu$-measure zero. Globally $\norm{\phi}_{C^0}\leq A'$.
	\end{itemize}
	Then for $0<s<s_0(\lambda, n, \alpha, A, A')\ll 1$, there is a uniform estimate 
	\[
	\sup_Y \phi \leq C(\lambda, n, \alpha, A, A') s.
	\]
\end{thm}

Theorem \ref{UniformL1stabilitythm} should be viewed as a one-sided version of $L^1$-volume stability estimate. The goal is to compare the two potentials $\phi$ and $0$, without loss of generality. The Skoda type estimate is the weakened version of the volume density integrability assumption (\ref{Lpdensity}) as before. Assume for the moment that this holds for \emph{both measures} $d\mu$ and $d\nu$. After adjusting $\phi$ by an additive constant, we might as well assume $\mu(E_0)\approx \nu(E_0)\approx \frac{1}{2}$, noticing that the total variation between the two measures is small by assumption. Then we can reverse the role of $\omega_\phi$ and $\omega$, to deduce a \emph{two-sided} smallness bound on $\phi$, which is the content of the $L^1$-volume stability estimate.




\subsection{Savin's small perturbation theorem}

Savin \cite{Savin} proved that for a large class of second order elliptic equations satisfying certain structural conditions, any viscosity solution $C^0$-close to a given smooth solution has interior $C^{2,\gamma}$-bound. In particular this applies to complex MA equation. Combined with the Schauder estimate,

\begin{thm}\label{Savin}
	Fix $k\geq 2$ and  $0<\gamma<1$.
	On the unit ball, let $v$ be a given smooth solution to the complex Monge-Amp\`ere equation $(dd^c v)^n=1$. Then there are constants $0<\epsilon\ll 1$ and $C$ depending on $n, k,\gamma, \norm{v}_{C^{k,\gamma}}$, such that if 
	\[
	(dd^c (u+v))^n=1+f, \quad \norm{f}_{C^{k-2,\gamma}}<\epsilon,
	\]
	and $\norm{u}_{C^0}<\epsilon$, then $\norm{u}_{C^{k,\gamma}(B_{1/2})}\leq C\epsilon$.
\end{thm}

Savin's theorem has fully nonlinear nature, because the perturbative machinery only applies once the solution has \emph{a priori} $C^2$ bound. His proof has two main parts: first he shows a Harnack inequality by a nontrivial application of  Aleksandrov-Bakelman-Pucci estimates, and then uses a compactness argument to prove $C^{2,\gamma}$ estimate, similar to De Giorgi's almost flatness theorem for minimal surfaces \cite{DeGiorgi}.

\subsection{Regularity theory for real Monge-Amp\`ere}\label{RegularitytheoryforrealMA}

There is an extensive literature on the local regularity theory for the real Monge-Amp\`ere equation, largely due to the Caffarelli school.  All results surveyed here can be found in \cite{Mooney}.

Any convex function on an open set $v: \Omega\subset \R^n\to \R$ has an associated Borel measure called the Monge-Amp\`ere measure, defined by
\[
MA(v)(E)= |\partial v (E)|,
\]
where $|\partial v(E)|$ denotes the Lebesgue measure of the image of the subgradient map on $E\subset \Omega$. Given a Borel measure $\mu$, a solution to $MA(v)=\mu$ is called an \emph{Aleksandrov solution} to 
$
\det(D^2 v)= \mu;
$ 
if $v\in C^2$, this is the classical real Monge-Amp\`ere equation. 
We shall assume a two-sided density bound
\[
\det(D^2 v)=f \text{ in }B_1, \quad 0<\Lambda_1\leq f\leq \Lambda_2.
\]
Let $B_1\setminus \Sigma$ be the set of strictly convex points of $v$, namely there is a supporting hyperplane touching the graph of $v$ only at one point. Then Caffarelli \cite{Caff1}\cite{Caff2}\cite{Caff4} shows
\begin{itemize}
	\item
	If $f\in C^\gamma(B_1)$, then $v\in C^{2,\gamma}_{loc}(B_1\setminus \Sigma)$. Then by Schauder theory, if $f$ is smooth, then $v$ is smooth in $B_1\setminus \Sigma$.
	\item
	If $L$ is a supporting affine linear function to $v$, such that the convex set $\{v=L\}$ is not a point. Then $\{v=L\}$ has no extremal point in the interior of $B_1$. 
	\item
	The above affine linear set $\{v=L\}$ has dimension $k< n/2$. 
\end{itemize}

Mooney \cite{Mooney} shows further that
\begin{itemize}
	\item The singular set $\Sigma$ has $(n-1)$-Hausdorff measure zero. Consequently $B_1\setminus \Sigma$ is path connected (because a generic path joining two given points does not intersect a subset of zero $(n-1)$-Hausdorff measure).

	\item  The solution $v\in W^{2,1}_{loc}(B_1)$ even if $\Sigma$ is nonempty.
\end{itemize}

\begin{rmk}\label{Mooneyremark}
	A classical counterexample of Pogorelov shows that for $n=3$, the singular set $\Sigma$ can contain a line segment. This is generalised by Caffarelli \cite{Caff4}, who for any $k<n/2$ constructs examples where $f$ is smooth but $\Sigma$ contains a $k$-plane. A surprising example of Mooney \cite{Mooney} shows that the Hausdorff dimension of $\Sigma$ can be larger than $n-1-\epsilon$ for any small $\epsilon$. This means the local regularity theory surveyed above is essentially optimal.
\end{rmk}

\subsection{Special Lagrangian fibration}\label{SpecialLagrangiansurvey}


The classical result of McLean says that the deformation theory of special Lagrangians with phase $\theta$ is unobstructed, and the first order deformation space is isomorphic to $H^1(L, \R)$. Thus if $L$ is diffeomorphic to $T^n$, then the deformation space is $n$-dimensional, compatible with the SYZ conjecture that $X$ admits a special Lagrangian $T^n$-fibration. A sufficient condition to construct Slag fibrations, under the very strong hypothesis of collapsing metric with locally bounded sectional curvature, is obtained by Zhang \cite[Thm 1.1]{Zhang}.

The essence of Zhang's result is a standard application of the implicit function theorem, and we shall summarize the key points (\cf \cite[section 4]{Zhang} for more details). Denote $Y_r= T^n\times B(0, r)\subset T^n_{x_i}\times\R^n_{y_i}\simeq T^* T^n$, where $r\gg 1$ is fixed. The trivial example of a special Lagrangian fibration is the following: the CY structure is the flat model
\[
g= \sum (dx_i^2+ dy_i^2), \quad \omega= \sum dx_i \wedge dy_i, \quad \Omega= \bigwedge (dx_j+ \sqrt{-1}dy_j),
\]
and the Slag fibration is just the projection to the $\R^n_{y_i}$ factor, namely the tori $T^n\times \{y\}$ are special Lagrangians. Zhang considers a family of CY structures $(g_k, \omega_k, \Omega_k)$ converging to $(g, \omega, \Omega)$ in the $C^\infty$-sense on $Y_{2r}$ (which follows from his bounded sectional curvature assumptions by elliptic bootstrap), such that $\omega_k\in [\omega]\in H^2(Y_{2r}, \R)$. Small deformations of the standard $T^n$ fibres can be represented as graphs on $T^n$: for $y\in \R^n$ and a 1-form $\sigma$ on $T^n$ orthogonal to the harmonic 1-forms $dx_1, \ldots dx_n$, write
\[
L(y, \sigma)= \text{Graph}( x\mapsto y+ \sigma(x)  ) \subset T^*T^n.
\]
The condition for $L(y, \sigma)$ to be a special Lagrangian with respect to $(g_k, \omega_k, \Omega_k)$ is
\begin{equation}\label{SlagZhang}
\omega_k|_{ L(y,\sigma) }=0, \quad \text{Im}( e^{\sqrt{-1}\theta_k} \Omega_k  )|_{ L(y,\sigma)  }=0,
\end{equation}
where $\theta_k$ are chosen so that $\int_{T^n}  e^{\sqrt{-1}\theta_k}\Omega_k >0$.  Zhang shows by perturbation arguments that for each $y\in B(0, \frac{3r}{2})$ and $k\geq k_0\gg 1$, there is a unique $\sigma=\sigma_{k,y}$ such that $L(y, \sigma_{k,y})$ solves (\ref{SlagZhang}) with small norm bound $\norm{\sigma_{k,y}}<\delta\ll 1$. He then uses another implicit function argument to show that these special Lagrangians indeed define a local special Lagrangian $T^n$-fibration on some open subset of $Y_{3r/2}$ containing $Y_r$.


\section{Nonarchimedean geometry}\label{NAsurvey}

Non-archimedean (NA) pluripotential theory is a close analogue of K\"ahler geometry. Impressionistically,

\begin{itemize}
	\item NA geometry offers a natural language to describe the degeneration of complex manifolds into real simplicial/tropical objects.
	
	\item It systematically encodes the combinatorics of tropical geometry.

	\item
	Usual notions in K\"ahler geometry such as functions, line bundles, K\"ahler metrics, complex Monge-Amp\`ere  measures, have natural (albeit exotic looking) analogues in NA geometry.

	\item 
	An analogue of the Calabi conjecture holds in the NA context: one can solve the NA Monge-Amp\`ere equation.

	\item 
	Under additional hypotheses, the NA Monge-Amp\`ere measure agrees with the real Monge-Amp\`ere measure.
\end{itemize}

We shall explain below that
the basic concepts of NA pluripotential theory can be motivated from the heuristic principle that \emph{NA geometry is the limit of complex geometry in the hybrid topology}. For some imprecise intuition, one may imagine non-archimedean geometry approximately as the SYZ base, and the hybric topology convergence roughly amounts to the collapse of an SYZ fibration to its base. Our persepctive is heavily influenced by Boucksom et al. \cite{Boucksom}\cite{Boucksom1}\cite{Boucksomsemipositive}\cite{Boucksomsurvey}.

\subsection{Berkovich space, hybrid topology}\label{Berkovichspace}

We mentioned in section \ref{Blowup} that  for a given polarized algebraic degeneration, the choice of snc models is highly non-unique.
There are two viewpoints on extracting invariant information:
 
 \begin{itemize}
 	\item In birational geometry, one aims to find optimal representatives via the minimal model program. Typically, this will leave the snc world, and require divisorial log terminal models \cite{NicaiseXu}\cite{NicaiseXuYu}, but the minimal models may still be non-unique.

 	\item In non-archimedean geometry, one looks simultaneously at the tower of all snc models, and take the formal limit of their dual complexes, known as the Berkovich space.

 \end{itemize}

Good references can be found in \cite[A]{KS}  \cite[Appendix]{Boucksom1}\cite[chapter 2,3]{Boucksomsemipositive}.

An insight of Berkovich is that by thinking of points as multiplicative seminorms, one obtains a kind of geometry analogous to complex manifolds.
Let  $K\simeq\C(\!(t)\!)$ be equipped with its standard absolute value $|\cdot |_0=e^{-ord_t}$ where $ord_t$ is the valuation defined by the vanishing order. Its ultrametric property 
\[
|f+g|_0\leq \max\{ |f|_0, |g|_0  \}
\]
gives the name `non-archimedean' to the subject. Let $X_K$ be a smooth, geometrically connected, projective scheme over $\text{Spec}(K)$; the main examples come from base changing an algebraic degeneration family $X$ over a punctured curve. Choose a finite cover of $X_K$ by affine open sets of the form $U=\text{Spec}(A)$, where $A$ is a finitely generated $K$-algebra. The space $U^{an}$ is defined as the set of all multiplicative seminorms $|\cdot|_x: A\to \R_{\geq 0}$ extending the absolute value of $K$, endowed with the weakest topology so that the function $x\mapsto |f|_x $ is continuous  for any $f\in A$. The \emph{Berkovich space} $X_K^{an}$ is then obtained by gluing together $U^{an}$; the notation stands for `analytification'. As a topological space $X_K^{an}$ is compact and Hausdorff. In the CY case, the point-set description of $X_K^{an}$ is meant to encode information about the base of the SYZ fibration; there is also a natural structure sheaf which encodes information about the complex structure
\cite{KS}.

Let $R\simeq \C[\![t ]\!]$. The concept of models over $\text{Spec}(R)$ is entirely analogous to the case over algebraic curves. The dual intersection complexes $\Delta_{\mathcal{X}}$ for snc models over $\text{Spec}(R) $ can be compared with $X_K^{an}$ through two natural maps:

\begin{itemize}
	\item There is a continuous \emph{embedding map} $emb: \Delta_{\mathcal{X} }\to X_K^{an}$. Writing $\mathcal{X}_0=\sum b_i E_i$, each divisor $E_i$ defines $val_{E_i}= \frac{ord_{E_i} }{ b_i  }$ through the vanishing order $ord_{E_i}$, so that $e^{- val_{E_i} }$ is a point in $X_K^{an}$, called a \emph{divisorial point}. More generally, given a point $x=(x_0,\ldots x_p)$ in the interior of a face $\Delta_J\subset \Delta_\mathcal{X}$ corresponding to $E_J=\cap_0^p E_i$, we can associate a \emph{quasi-monomial valuation}: expanding any local function $f$ around $E_J$ in Taylor series,
	\[
	f=\sum_{ \alpha\in \N^{p+1} } f_\alpha z_0^{\alpha_0}\ldots z_p^{\alpha_p}, \quad f_\alpha\in K(E_J)
	\]
	then the quasi-monomial valuation is
	\[
	val_x(f) = \min \{  \sum_0^p \alpha_i x_i | f_\alpha\neq 0       \}.
	\] 
	Thus $x$ gives rise to a point $e^{-val_x}\in X_K^{an}$. We shall regard $\Delta_{\mathcal{X}}$ as a subset of $X_K^{an}$. In particular, the essential skeleton embeds into $X_K^{an}$.

	\item
	There is a continuous \emph{retraction map}  $r_{ \mathcal{X} }: X_K^{an}\to \Delta_{\mathcal{X} } $, which restricts to the identity on $\Delta_{  \mathcal{X}}\subset X^{an}$. 
	Any point $e^{-v}\in X_K^{an}$ admits
	a center on $\mathcal{X}$. 
	This is the unique scheme theoretic point
	$\xi \in X_0$ such that 
	$|f|_x \leq 1$ for $f\in \mathcal{O}_{\mathcal{X},\xi}$
	and $|f|_x < 1$ for
	$f\in m_{\mathcal{X},\xi}$.
	Let $J \subset I$ 
	be the maximal subset such that
	$\xi\in E_J$. Then
	$r_{ \mathcal{X} }(x) \in \Delta_{\mathcal{X} }$ 
	corresponds to the quasi-monomial valuation with the same value for
	$-\log |z_j|_x, j \in J$. Concretely, one should think the retraction map is about reading off logarithmic coordinates. 
	
	For another perspective, if $\mathcal{X}'$ is a blow up of $\mathcal{X}$, then there is a natural simplicial map $\Delta_{\mathcal{X}'}\to \Delta_{\mathcal{X}}$, which is identity on $\Delta_{\mathcal{X}}\subset \Delta_{\mathcal{X}'}$. The retraction map $X_K^{an}\to \Delta_{\mathcal{X}}$ can be viewed as a formal limit for very large $\mathcal{X}'$.

\end{itemize}

\begin{rmk}
	The retraction map depends on the choice of the \emph{model}. There are examples where two models $\mathcal{X}$ and $\mathcal{X}'$ define the same $\Delta_{\mathcal{X}}$ as a subset of $X_K^{an}$, but the retraction maps are different \cite[Appendix]{GublerJill}.
\end{rmk}

With these comparison maps, the Berkovich space $X_K^{an}$ is homeomorphic to the inverse limit of the dual intersection complexes of the snc models:
\[
X_K^{an}\simeq  \varprojlim_{\text{snc models} } \Delta_{ \mathcal{X} }
\]
Conceptually, \emph{an snc model gives a finite approximation of the Berkovich space}.


We now indicate how NA geometry is unified with complex geometry. Consider an algebraic degeneration $X$ over a punctured curve. Let $|\cdot|$ denote the usual absolute value for complex numbers. Given a $\C$-point $z\in X_t$ for $0<|t|\ll 1$, inside some affine chart $U=\text{Spec}(A)$ of $X$, we can define a multiplicative seminorm $A\to \R_{\geq 0}$ (not non-archimedean!) 
\begin{equation}\label{hybridconvergence}
f\mapsto e^{- \log |f(z)|/\log |t|   }=|f(z)|^{1/|\log |t| |}.
\end{equation}
As a sequence of points $z$ move towards $t\to 0$, for any given meromorphic function $f=\sum a_k t^k$ on the base, $\lim_{t\to 0}\log |f(z)|/\log |t|= ord_0(f)$ which is the standard NA valuation on $K$. Thus the points on $X_K^{an}$ are natural limits of the multiplicative seminorms defined by $\C$-points on $X_t$. One can formalize this notion by introducing a \emph{hybrid topology} on $X\sqcup X_K^{an} $, so that $X_K^{an}$ takes the place of the central fibre \cite[Appendix]{Boucksom1}. The functions $f\in A$ then induce local continuous functions on $X\sqcup X_K^{an} $.

The `hybrid' space $X\sqcup \Delta_{\mathcal{X}}$ discussed in section \ref{volumeasymptoteessentialskeleton} can be understood as a finite approximation. Given an snc model $\mathcal{X}$, and take a sequence of $\C$-points $q_k$ tending to $e^{-v}\in X^{an}$, whose image under the retraction map $r_{\mathcal{X} }$ is $x=(x_0,\ldots x_p)\in \Delta_J\subset \Delta_{\mathcal{X}}$.  Tautologically $q_k$ concentrate near $E_J$, and in the local coordinates $z_0,\ldots z_p$, we have $\log |z_i(q_k)|/\log |t|\to v(z_i)= x_i $, which is equivalent to  $\text{Log}_{\mathcal{X} } (z_k)\to x=(x_0,\ldots x_p)\in \Delta_J\subset \Delta_{ \mathcal{X} }$. Formally, the topology on $X\sqcup X_K^{an}$ is the inverse limit of $X\sqcup \Delta_{\mathcal{X}}$ by taking all snc models.


\subsection{Model functions, metrics, positivity}

We now discuss functions, line bundles, and metrics on $X_K^{an}$ \cite{Boucksomsemipositive}.
Given a model $\mathcal{X}$ over $\text{Spec}(R)$ and a Cartier divisor $D$ supported on the central fibre $\mathcal{X}_0$, we can associate a continuous function on $X_K^{an}$ by setting
\[
\phi_D(x)= \max\{  \log |f|_x : f\in \mathcal{O}_{\mathcal{X}  }(D)       \},
\]
The association $D\mapsto \phi_D$ extends by $\Q$-linearity.
Functions obtained in the $\Q$-span using all such choices of models and divisors are called \emph{model functions} on $X_K^{an}$, which form a dense subset of $C^0(X_K^{an})$. The restrictions of such functions to dual intersection complexes are \emph{piecewise affine}.

To understand the complex geometric meaning, we think of models base changed from snc models over an algebraic curve $S$. The divisor $D$ prescribes a class of functions $\phi$ on the total space of the snc model with \emph{analytic singularities}: 
\[
\phi= \log |f|+ C^\infty \text{ function},
\]
where $f$ is a local defining function of $D$. When we consider the rescaling of the restrictions to $X_t$
\[
\phi_t= \frac{1}{ \log |t|  } \phi|_{X_t},
\]
only the singular term is relevant in the limit $t\to 0$, and $\phi_t$ converge to $-\phi_D$ in the hybrid topology.

We think about line bundles on $X_K^{an}$ via the GAGA principle: the line bundles on $X_K^{an}$ correspond to the line bundles $L$ on the scheme $X_K$. A \emph{continuous metric} on $L$ assigns to each local section $s$ a nonnegative continuous  local function $\norm{s}$ on open subsets of $X_K^{an}$, compatible with the sheaf structure, such that $\norm{fs}(x)= |f|_x\norm{s}(x)$, and $\norm{s}>0$ if $s$ is a local frame of $L$.
Given a continuous metric, any other continuous metric on $L$ is of the form $\norm{\cdot} e^{-\phi}$ for some $\phi\in C^0(X^{an})$, analogous to the usual relation between Hermitian metrics and K\"ahler potentials. As such $\phi$ is referred to as a \emph{potential} function.

Given a model $\mathcal{X}$ for $X_K$, a model $\mathcal{L}$ of $L$ is a line bundle $\mathcal{L}\to \mathcal{X}$ with $\mathcal{L}|_X=L$. To this data we can associate a unique metric $\norm{\cdot}_{ \mathcal{L} }$ on $L$ with the following property: if $s$ is a nowhere vanishing local section of $\mathcal{L}$ on an open set $\mathcal{U}\subset \mathcal{X}$, then $\norm{s}_{ \mathcal{L} } \equiv 1$ on $\mathcal{U}\cap X_K$. This is well defined because any two such sections differ by the multiplication of an invertible function, whose NA absolute value equals the constant one. One can extend this construction to $\Q$-line bundles, and the metrics arising this way are called \emph{model metrics}. They are dense within the continuous metrics.

To see the complex geometric interpretation, we imagine a line bundle $\mathcal{L}$ on some snc models $\mathcal{X}$ over an algebraic curve. Equip $\mathcal{L}$ with any smooth Hermitian metric $h$. Given a local section $s$ of $\mathcal{L}$, the prescription compatible with (\ref{hybridconvergence}) is to consider the local functions on $X_t$ 
\[
z\mapsto  |s(z)|_h^{ 1/|\log |t||  }.
\]
Taking the limit as $t\to 0$, we precisely get the model metric. Notice the ambiguity in the choice of the Hermitian metric is obliterated in the limit.

A paramount notion in K\"ahler geometry is the \emph{positivity} of the metric, usually phrased in terms of psh properties of the potential. The above discussion suggests that in the NA setting, namely $t\to 0$, such a notion should be expressible as a numerical property of the line bundle.

\begin{Def}\cite[Thm. 2.17]{Boucksom}
	(Semipositivity I) 
	Let $\norm{\cdot}$ be a model metric on $L$, associated to a $\Q$-line bundle $\mathcal{L}$ on a model $\mathcal{X}$ of $X_K$. Then
	\begin{itemize}
		\item the metric $\norm{ }$ is a semipositive model metric iff $\mathcal{L}$ is nef, namely $\mathcal{L}\cdot C\geq 0$ for any projective curve $C$ contained in $\mathcal{X}_0$;
		\item
		a continuous metric $\norm{ } e^{-\phi}$ is \emph{semipositive} iff it is the uniform limit of some sequence of semipositive model metrics on $X_K^{an}$.
	\end{itemize}
\end{Def}

\begin{rmk}
	The advantage of `nef' instead of `ample' is that if we blow up the model further, the pullback of the model line bundle will stay nef, but ampleness will be lost.	
\end{rmk}

In K\"ahler geometry the definition of psh function is local in the complex charts. Since the dual intersection complexes are simplicial objects, one would expect the NA analogous notion to be related to convex functions. This intuition is partially valid:

\begin{prop}\label{convexityonfaces}
	\cite[Prop 5.9]{Boucksomsemipositive} Let $\mathcal{X}$ be an snc model for $X_K$, and $\mathcal{L}\to \mathcal{X}$ be a model line bundle for $L\to X$, with associated closed (1,1)-form $\theta$. Then the restriction of any continuous $\theta$-psh function to any face of $\Delta_{\mathcal{X} }\subset X_K^{an}$ is convex.
\end{prop}

The picture is that general $\theta$-psh functions define convex functions on the faces of $\Delta_{ \mathcal{X} }$, and among them the $\theta$-psh model functions give piecewise affine approximations with finer and finer grids.

\subsection{Approximation by Fubini-Study metrics}\label{FubiniStudyapproximation}

A fundamental result in K\"ahler geometry is that any K\"ahler metric in an integral class can be approximated by Fubini-Study metrics associated with projective embeddings. While the usual Fubini-Study metric depends on a choice of a Hermitian inner product on the $\C$-vector space of global sections, the NA analgoue depends on a NA norm on the $K$-vector space $V=H^0(X_K,m L)$ for $m\gg 1$, with the ultrametric property $\norm{x+y}_V\leq \max\{ \norm{x}_V, \norm{y}_V \}$. In our case $K=\C(\!(t)\!)$ one can select a $K$-basis $s_0, s_1, \ldots s_N$ for $V$, such that
\[
\norm{ a_0s_0+ \ldots+ a_N s_N}_V=\max\{ |a_0|\norm{s_0}_V,\ldots ,|a_N|  \norm{s_N}_V  \}, \quad \forall a_i\in K.
\]
The NA Fubini-Study metric on $L\to X_K$ can be defined as
\[
\norm{s}_{FS}(x)=  \inf_{ \tilde{s}\in V, \tilde{s}(x)=s^{\otimes m}(x)} \norm{ \tilde{s}}_V^{1/m}, \quad \forall x\in X_K^{an}.
\]
Concretely in the orthogonal basis, written in a local trivialisation,
\begin{equation}\label{FubiniStudyNA}
\norm{s}_{FS }(x)= \frac{ |s(x)|}{ \max_j \{  |s_j(x)|/\norm{s_j}_V  \}^{1/m}   }, \quad \forall x\in X^{an}.
\end{equation}
A NA analogue of the \emph{Fubini-Study approximation theorem} gives an alternative view on semipositive metrics:

\begin{prop}
	(Semipositivity II)
	\cite{ChenMoriwaki} Assume $L\to X_K$ is ample. Then a continuous metric on $L$ is semipositive iff it can be written as a uniform limit of Fubini-Study metrics.
\end{prop}

\begin{rmk}
In the approximation theorem we may assume $s_i$ to be finite Laurent polynomials.
\end{rmk}

For the complex geometric interpretation, we assume as usual $X_K$ is the base change of an algebraic degeneration family $X$, with an ample polarisation line bundle $L$. For any given NA Fubini-Study metric (\ref{FubiniStudyNA}), we can associate a family of Fubini-Study metrics on $(X_t, L)$: 
\begin{equation}\label{FubiniStudyXt}
\norm{s}_{FS,t}(z)=  \frac{ |s(z)|}{ \{ \{ \sum_j  |s_j(z,t) |^2 |t|^{2\log \norm{s_j}_V}  \}^{1/2m}   }, \quad \forall z\in X_t.
\end{equation}
Here $s_j$ make sense for finite $t$ because they are selected as finite Laurent polynomials in $t$. As $t\to 0$, the Fubini-Study metrics converge to the NA analogue, or more precisely $\norm{s}_{FS,t}^{1/|\log |t||  }$ converges to (\ref{FubiniStudyNA}) in the hybrid topology on $X\sqcup X_K^{an}$.

\subsection{NA Monge-Amp\`ere measure}\label{NAMAmeasure}

The NA Monge-Amp\`ere measure \cite{ChambertLoir} is defined through \emph{intersection theory} in a somewhat counterintuitive manner. As a motivation, we consider the complex analytic setting of an snc model $\mathcal{X}$ over an algebraic curve, equipped with a Hermitian line bundle $(\mathcal{L}, h)$ with curvature form $\theta$ in the class $c_1(\mathcal{L})$. Then $\theta^n$ defines a family of $n$-forms on $X_t$, such that $\int_{X_t}\theta^n$ equals the intersection number $(L^n)$. The question is to describe the limit of these $n$-forms, when we view $X_t$ as converging to the dual intersection complex $\Delta_{\mathcal{X} }$ (\cf section \ref{volumeasymptoteessentialskeleton}).

We write $X_0=\sum_{i\in I} b_i E_i$.
Recall that the regions on $X_t$ corresponding to the faces in the dual intersection complex are from the algebraic perspective only small neighbourhoods of $E_J$. Thus the limit of $\theta^n|_{X_t}$ can only be supported at the vertices of $\Delta_{\mathcal{X}}$, which correspond to the components $E_i$. The amount of delta masses concentrated at the vertices are
\[
b_i\int_{E_i} \theta^n= b_i\mathcal{L}^n\cdot E_i,
\]
where $b_i$ appears due to the multiplicity of the sheets. Reassuringly,
\[
\sum_i b_i\mathcal{L}^n\cdot E_i= (L^n)
\]
gives the correct total mass.

Back to the NA setting, 
given a model $\Q$-line bundle $\mathcal{L}\to \mathcal{X}$ for  $L\to X_K$, we write $\mathcal{X}_0= \sum_i b_i E_i$, and denote the divisorial points associated to $E_i$ as $q_i$. We
can then \emph{define} the NA Monge-Amp\`ere measure for the model metric $\norm{\cdot}_ {\mathcal{L} }$ as the following signed atomic measure supported at $q_i\in X_K^{an}$:
\[
MA( \norm{\cdot}_{ \mathcal{L} } )= \sum_{E_i} b_i (\mathcal{L}^n\cdot E_i) \delta_{q_i}
\] 
This definition is compatible with pullback of line bundles by the projection formula, and ensures the total mass is the intersection number $(L^n)$. If $\norm{\cdot}_{\mathcal{L}  }$ is furthermore semipositive, then the intersection numbers are non-negative, so $MA( \norm{\cdot}_{ \mathcal{L} } )$ is a measure.

The theory of NA MA measures bears  strong resemblance to the complex MA measures \cite{Boucksom}\cite{Boucksomsurvey}:

\begin{itemize}
\item In the complex analytic world, one first define the complex MA for smooth potentials. A general continuous semipositive potential in a K\"ahler class is the uniform limit of smooth potentials, and its complex MA measure is then determined by the weak continuity under $C^0$-convergence.

\item In the NA world, one first define the NA MA measure for the model metrics. A general continuous semipositive metric on $L$ is the uniform limit of a sequence of continuous semipositive model metrics \cite[Cor. 8.8]{Boucksomsemipositive}, and its NA MA measure can be defined as the unique limiting Radon measure of the NA MA measures for the sequence \cite[Cor. 3.5]{Boucksom}.

\end{itemize}

Their main difference lies in the highly nonlocal appearance of the NA MA measure. The recent result of Vilsmeier \cite{Vilsmeier} offers a more concrete perspective:

\begin{prop}\label{NAMArealMAcomparisonProp}
	(NA MA-real MA comparison) Let $(\mathcal{X},\mathcal{L})$ be a semistable snc model of $(X_K,L)$, and $\text{Int}(\Delta_J)$ be an $n$-dimensional open face of $\Delta_{\mathcal{X} }$. Recall the retraction map $r_{\mathcal{X}}: X_K^{an}\to \Delta_{\mathcal{X} }$. Let $\phi\in C^0(X_K^{an})$ be the potential of a semipositive metric $\norm{\cdot}_{ \mathcal{L} }e^{-\phi}$, and suppose $\phi=\phi\circ r_{\mathcal{X}}$ on $r_{\mathcal{X}}^{-1}(\Delta_J )$,   then on $\text{Int}(\Delta_J)$ the pushforward of the NA MA measure
	\[
	r_{\mathcal{X}* } MA( \norm{\cdot}e^{-\phi} )= n! MA_\R (\phi|_{\text{Int}(\Delta_J)})
	\]
	equals the real MA measure of the convex function $\phi|_{\text{Int}(\Delta_J)}$ up to a factor $n!$.
\end{prop}

The rigorous proof of this comparison uses intersection theory, and the following is a heuristic explanation. Consider an snc model $\mathcal{X}$ over an algbebraic curve as in the motivation, and assume furthermore that it is semistable. Recall our heuristic dictionary that a metric $\norm{\cdot}$ on $L\to X_K$ should encode a family of Hermitian metrics $h_t$ on $L\to X_t$, such that $h_t^{ 1/|\log |t||  }\to \norm{\cdot}^2$ in the hybrid topology, and the NA MA measure of $\norm{\cdot}$ should be the limit of the measures associated to the curvature forms of $h|_{X_t}$. We now focus on the neighbourhood of an $n$-dimensional open face $\text{Int}(\Delta_J)\subset \Delta_{ \mathcal{X} }\subset  \Delta_{ \mathcal{X} }\sqcup X  $, where we have local coordinates $z_0,\ldots z_n$ with $\prod_0^n z_i=t$, and $x_i=\frac{\log |z_i|}{\log |t|}$. In the local picture we identify metrics with potentials, so $\norm{\cdot}\sim e^{-\phi}$, and after ignoring $C^0$-fluctuation effects $h_t^{  1/|\log |t||}\sim e^{-  2\phi\circ \text{Log}_{\mathcal{X}} }$.
Imposing more smoothness assumptions, the curvature form of $h_t$ is approximately
\[
|\log |t|| dd^c\phi\circ\text{Log}_{ \mathcal{X}}= \frac{-1}{2\pi} \sum_{1\leq i,j\leq n} \frac{\partial^2\phi}{\partial x_i \partial x_j} dx_i \wedge d\text{arg}(z_j).
\]
The NA MA measure should agree with the limiting pushforward measure
\[
\lim_{t\to 0}\text{Log}_{ \mathcal{X}*} (|\log |t|| dd^c\phi\circ\text{Log}_{ \mathcal{X}} )^n = n!\det(D^2 \phi) |dx_1\ldots dx_n|=n! \text{MA}_\R(\phi)
\]
which equals the real MA measure up to the factor $n!$.

\begin{rmk}
	In this heuristic calculation, the assumption for $\phi$ to factor through the retraction map allows us to replace the hybrid space $X\sqcup X_K^{an}$ by its finite approximation $X\sqcup \Delta_{\mathcal{X} }$.
\end{rmk}

\subsection{NA Calabi conjecture}\label{NACalabi}

The central result of NA pluripotential theory is the solution to the NA analogue of the Calabi conjecture. A good survey is \cite{Boucksomsurvey}.

\begin{thm}
	\cite{Boucksom} Let $X_K$ be a smooth projective K-scheme arising from the base change of an algebraic degeneration family. Let $L$ be an ample line bundle on $X_K$, and $d\mu$ be a Radon probability measure supported on the dual intersection complex of some snc model of $X_K$. Then there is a unique continuous semipositive metric $\norm{\cdot}$ on $L$, such that
	\[
	MA( \norm{\cdot })= (L^n) d\mu.
	\]

\end{thm}

Their strategy uses a \emph{variational method}. There is a concave \emph{energy functional} $\mathcal{E}$ on the space of continuous semipositive metrics on $L$ (equivalently viewed as continuous $\theta$-psh potentials $\phi$), whose first variation is given by the NA MA measure. One seeks a maximizer of the functional
\[
F_\mu(\phi)= \mathcal{E}(\phi)-(L^n)\int_{X_K^{an}} \phi d\mu,
\]
by first enlarging the space of $\phi$ to a function space $PSH(X_K^{an}, \theta)$ which is \emph{compact} modulo the addition of a real constant; this is analogous to the $L^1$-compactness of $PSH(X,\omega)/\R$ in the K\"ahler setting. The notions of the NA MA measure and the energy functional extend naturally to the \emph{energy class} functions inside $PSH(X_K^{an}, \theta)$, much like in the complex pluripotential theory setting. One then shows the maximizer is in fact a \emph{critical point}, namely a weak solution to the NA MA equation. This is subtle since small perturbations of functions in $PSH(X_K^{an}, \theta)$ may fall outside of the class by losing positivity. One proves the continuity of the weak solution using analogues of Kolodziej's estimates. The uniqueness of the solution again relies on the concavity of $\mathcal{E}$.

While this strategy shares a very similar logical structure with the complex analytic setting, the technical foundations are built upon intersection theory and vanishing theorems in birational geometry, instead of differential operators.

The main case of interest to us is when $X_K$ arises from a large complex structure limit. Then NA pluripotential theory provides a unique solution to
\begin{equation}\label{NAMACY}
MA( \norm{\cdot }_{CY})= (L^n) d\mu_0,
\end{equation}
where $d\mu_0$ is the Lebesgue measure supported on the essential skeleton $Sk(X)\subset X_K^{an}$ (\cf section \ref{volumeasymptoteessentialskeleton}). We call $\norm{\cdot}_{CY}$ the \emph{non-archimedean Calabi-Yau metric}.

\subsection{Comparison property}\label{Comparisonproperty}

Very little is proven about the non-archimedean CY metric $\norm{\cdot}_{CY}$ beyond existence and continuity. We now discuss the meaning of the following conjectural NA MA-real MA comparison property.

\begin{Def}
	We say $\norm{\cdot }_{CY}$ satisfies the \emph{NA MA-real MA comparison property}, if there exists a semistable snc model $(\mathcal{X}, \mathcal{L})$ of $(X,L)$ with the property that, the potential $\phi_0$ defined by $\norm{\cdot}_{ CY }= \norm{\cdot}_{ \mathcal{L}}e^{-\phi_0}$ satisfies $\phi_0=\phi_0\circ r_{\mathcal{X}  }$ on the preimages of the retraction map over all the $n$-dimensional open faces 
	$\text{Int}(\Delta_J)\subset Sk(X)$.

\end{Def}

Notice $\text{Int}(\Delta_J)\subset \Delta_{\mathcal{X} }$ inherits a natural integral affine structures.  
Since the restriction of $\phi_0$ is convex on these faces by Prop. \ref{convexityonfaces}, its real MA measure makes sense, and by Prop. \ref{NAMArealMAcomparisonProp} it satisfies the \emph{real MA equation} on $\text{Int}(\Delta_J)$
\begin{equation}\label{realMACY}
\text{MA}_\R(\phi_0)= \frac{ (L^n) }{ n!} d\mu_0.
\end{equation}

A few comments are in order:

\begin{itemize}
\item  The comparison property is a conjecture in algebraic/non-archimedean geometry, and does not \emph{a priori} involve PDE concepts. Its PDE implications come \emph{a posteriori}.

\item  The Kontsevich-Soibelman picture (\cf section \ref{KontsevichSoibelmanconj}) expects that there is a solution of the real MA equation on the essential skeleton, away from some singular locus. The NA-MA equation via the  comparison property is the only known systematic method to produce solutions.

\item In the context of \emph{toric invariant metrics} on toric varieties, there are comparison results between NA MA measure and real MA equation, \cf \cite[Prop. 4.4.4]{Gil}.

\item The reader may feel that NA pluripotential theory is a very long-winded way to solve the real MA equation. However, surprisingly enough, it is not even known how to formulate the real MA equation globally on $Sk(X)$ in general, not just on the $n$-dimensional faces but also on the lower dimensional faces.

One problem is that $Sk(X)$ does not come with an obvious preferred affine structure, but only a piecewise affine structure, so there is no obvious coordinate independent definition of the real MA measure. It seems that the affine structure conjectured by Kontsevich and Soibelman would need to be solved simultaneously with the real MA equation, rather like free boundary PDE problems.

 Another problem is that solving the real MA equation requires first specifying the class of convex functions to be admitted as potentials, just like solving the complex Monge-Amp\`ere equation requires first specifying the meaning of K\"ahler potentials. We do not currently know any direct way of defining the class of convex potentials on piecewise affine manifolds such as $Sk(X)$. The semipositive metrics on the Berkovich space $X_K^{an}$, abstract as it may be, is our only available substitute.

\item If one believes the NA Calabi-Yau metric $\norm{\cdot}_{CY}$ is the potential theoretic limit of the Calabi-Yau metrics on $X_t$ in the hybrid topology, and that the information of $\norm{\cdot}_{CY}$ can be recovered from data on $Sk(X)$, then one may be inclined to think that the potential of $\norm{\cdot}_{CY}$ factors \emph{globally} through some retraction map $X_K^{an}\to Sk(X)$, defined perhaps through some divisorial log terminal minimal model.

Such a statement would need to confront the difficulty that the divisorial log terminal model is not necessarily unique, and in principle the retraction map depends on the choice of the model. It seems highly nontrivial how the NA MA equation would select a preferred retraction map.

The formulation of the comparison property is more cautious than this. We allow the dual intersection complex $\Delta_{\mathcal{X}}$ to be strictly bigger than $Sk(X)$, and there is no assumption on the complement of the $n$-dimensional faces of $Sk(X)$. Regarding the problem above, if one is undecided between a finite number of candidate retraction maps, then one can pass to a common snc resolution (and perhaps pass to finite base change, to find a semistable snc resolution). Of course, the more we blow up the model $\mathcal{X}$, the weaker is the comparison property.

\begin{rmk}
The very recent work of Pille-Schneider and Mazzon \cite{Leonard} proposes gluing the retraction maps associated to several divisorial log terminal models to obtain a  map $X_K^{an}\to Sk(X)$. Their map still factors through the dual intersection complex of some larger snc model, hence is compatible with the comparison property.
\end{rmk}

\item Without the comparison property, it seems hard to give any differential geometric interpretation to $\norm{\cdot}_{CY}$ at all, since $X_K^{an}$ contains arbitrarily large dual intersection complexes, and thus is highly complicated. The heuristic intuition of this hypothetical scenario, is that the potential theoretic limit of the Calabi-Yau metrics would require \emph{infinitely many blow ups} to describe. This is not yet ruled out by a theorem; we leave the reader to judge its plausibility.

\end{itemize}

The open question for algebraic geometers is

\begin{Question}
Can the comparison property be proven for a sufficiently large class of examples, such as those from the Gross-Siebert program \cite{Gross}? 
\end{Question}

\section{Glimpse of proof strategy}

We discuss some recent progress on the weak metric version of the SYZ conjecture.

\begin{thm}\label{NAthm}
\cite{LiNA} Let $X\to S\setminus \{ 0\}$ be a large complex structure limit of Calabi-Yau manifolds, with the polarization ample line bundle $L\to X$. Assume the NA MA-real MA comparison property holds for $X$ (\cf section \ref{Comparisonproperty}). 
For sufficiently small $t\in S\setminus \{0\}$, there exists a special Lagrangian $T^n$-fibration with respect to the Calabi-Yau structure $(\omega_{CY,t}, \Omega_t)$ on an open subset of $X_t$ whose normalized Calabi-Yau measure tends to $100\%$ as $t\to 0$.
\end{thm}

There is a very particular family of projective Calabi-Yau hypersurfaces, for which the NA MA-real MA comparison property can be bypassed, at the cost of passing to subsequences:
\begin{equation}
X_t= \{  Z_0Z_1\ldots Z_{n+1}+ t \sum_{i=0}^{n+1} Z_i^{n+2}=0     \}\subset \mathbb{CP}^{n+1}, \quad t\in \R, 0<t\ll 1.
\end{equation}
We call this the \emph{Fermat family}, on account of the famous Fermat polynomial $\sum_{i=0}^{n+1} Z_i^{n+2}$.

\begin{thm}\cite{LiFermat}\label{Fermatthm}
For the Fermat family, consider the Calabi-Yau metrics $\omega_{CY,t}$ on $X_t$ in the polarisation class $\frac{1}{|\log |t||} c_1( \mathcal{O}(1))$. Then for a subsequence of $X_t$ as $t\to 0$, there exists a special Lagrangian $T^n$-fibration on an open subset of $X_t$, whose normalized Calabi-Yau measure tends to $100\%$.	
\end{thm}

Our exposition will focus on the main line of thought and its many subtleties, but not the full details of the proofs.

\subsection{Reduction to potential estimates}\label{Reductiontopotentialestimate}

The common part of the strategy is to reduce the existence question of special Lagrangians to $C^0$-estimate on the potential.

Given a fixed snc model $\mathcal{X}\to S$, there is a logarithm map $\text{Log}_{\mathcal{X}}: X_t\to \Delta_{\mathcal{X}}$ defined up to $O(\frac{1}{|\log |t||})$ coordinate ambiguity. Given an $n$-dimensional face $\Delta_J$ of $Sk(X)\subset \Delta_{\mathcal{X}}$, we consider the preimage $U_{J,t}\subset X_t$ under the logarithm map, of a slightly shrinked version of the interior of $\Delta_J$.  We will take the liberty of shrinking $U_{J,t}$ several times, as long as the deleted sets have negligible Calabi-Yau measure in the $t\to 0$ limit. Since $U_{J,t}$ can be regarded as a torus invariant subset of $(\C^*)^n$, we can make sense of $C^k_{loc}$ norms uniformly in $t$, by passing to the universal cover with the coordinates $\frac{\log z_i}{\log |t|}$.

The first main step is to improve $C^0$-estimate to $C^\infty$-estimate.

\begin{prop}\label{C0toCinfty}
(\cf \cite[section 4.5]{LiNA})  Let $\phi_0$ be an Alexandrov solution of the real MA equation (\ref{realMACY}) on the interior of $\Delta_J$. Suppose  the Calabi-Yau metrics on $U_{J,t}$ admit local potential functions $\phi_{CY, J,t}$ such that $\omega_{CY,t}=dd^c \phi_{CY, J,t}$, and $\norm{\phi_{CY,J,t}- \phi_0\circ \text{Log}_{\mathcal{X}}}_{C^0} \to 0$ as $t\to 0$. Then after slightly shrinking $U_{J,t}$, we have the $C^\infty$-asymptote
$
\norm{\phi_{CY,J,t}- \phi_0\circ \text{Log}_{\mathcal{X}}}_{C^k_{loc}} \to 0.
$

\end{prop}

\begin{proof}
	(Sketch)
\begin{itemize}
\item    The first ingredient is that by the regularity theory of real MA equation (\cf section \ref{RegularitytheoryforrealMA}), after deleting a subset of $\text{Int}(\Delta_J)$ of Hausdorff $(n-1)$-measure zero, then $\phi_0$ is smooth. After a slight shrinking of the remaining open set, then $\phi_0$ has $C^k$ bounds.

\item   The second ingredient is Savin's small perturbation theorem (\cf section \ref{Savin}). After passing to the local universal cover, both $\phi_{CY,J,t}$ and $\phi_0\circ \text{Log}_{\mathcal{X}}$ solve a complex Monge-Amp\`ere equation. The difference in their RHS vanishes in the $t\to 0$ limit in arbitrarily high $C^k$ norm, as a consequence of the volume form asymptote in section \ref{volumeasymptoteessentialskeleton}. Savin's result then improves the $C^0$ closeness of $\phi_{CY,J,t}$ and $\phi_0\circ \text{Log}_{\mathcal{X}}$ to $C^\infty_{loc}$ closeness, after small shrinking of $U_{J,t}$.

\end{itemize}
\end{proof}

Prop. \ref{C0toCinfty} implies the  \textbf{semiflat metric asymptote} on the slightly  shrinked $U_{J,t}$, with $C^\infty$ small error in the $t\to 0$ limit:
\begin{equation}
\omega_{CY,t}\sim  dd^c (\phi_0\circ \text{Log}_{\mathcal{X}})=\frac{\sqrt{-1}}{4\pi |\log |t||^2}\sum_{i,j}\frac{\partial^2 \phi_0}{\partial x_i\partial x_j} d\log z_i\wedge d\overline{\log z_j}.
\end{equation}
In terms of the Riemannian metric tensors,
\begin{equation}\label{CYmetricasymptote}
g_{CY,t}\sim \frac{1}{2\pi|\log |t||^2} \text{Re}\{  \sum_{1\leq i,j\leq n} \frac{\partial^2 \phi_0}{\partial x_i\partial x_j}  d\log z_i\otimes d\log \bar{z}_j \}.
\end{equation}
In particular the Riemannian curvature stays uniformly bounded in $U_{J,t}$. By the perturbation theory of special Lagrangians reviewed in section \ref{SpecialLagrangiansurvey}, the $T^n$ fibres of the logarithm map can be made into a \textbf{special Lagrangian fibration} by a $C^\infty$ small perturbation, on a slightly shrinked subset.

Suppose the assumption of Prop. \ref{C0toCinfty} holds on all the $n$-dimensional faces of $Sk(X)$, then the union of all $U_{J,t}$ cover almost all the CY measure on $X_t$, and the measure lost in the domain shrinking process is negligible. The weak metric version of the SYZ conjecture then follows.

\begin{rmk}
A subtlety  is that the \emph{local regularity} theory of real Monge-Amp\`ere equation allows for Hausdorff codimension $1+\epsilon$ singularities. This means the codimension two singularity prediction in the Kontsevich-Soibelman conjecture cannot follow simply from the above argument. One must find a more \emph{global} argument on $Sk(X)$, not just on the interior of its $n$-dimensional faces. 
\end{rmk}




\subsection{Strategy I: non-archimedean geometry}\label{StrategyI}

The remaining task is to achieve the local $C^0$-convergence of local potentials to a solution of the real MA equation on the open $n$-dimensional faces of $Sk(X)$ (\cf Prop. \ref{C0toCinfty}). The first strategy \cite{LiNA} is:

\begin{itemize}
\item Solve the real MA equation on $Sk(X)$, independent of the CY metrics on $X_t$.

\item Then attempt to compare the solution with the potential of the CY metrics on $X_t$.
First, one needs to produce a K\"ahler metric on $X_t$ whose local potential is $C^0$-close to the real MA solution on $Sk(X)$ in some topology. Then one needs some version of the $L^1$-volume stability estimate (\cf section \ref{Toolsfrompsh}) to show the $C^0$-smallness of the relative potential between this K\"ahler metric and the CY metric, at least in the generic region.
\end{itemize}

\subsubsection{Motivation for NA geometry}

The above strategy contains many problems:

\begin{itemize}
\item  As discussed in section \ref{Comparisonproperty}, it is unknown how to directly formulate the real MA equation on $Sk(X)$, nor do we know the precise class of convex functions needed for such formulations.

\item The essential skeleton is a simplicial complex, and $X_t$ is a complex manifold. These are conceptually very different objects, and we need a topology to unify both sides.

\item  Pluripotential theoretic arguments require the \emph{global positivity} (\ie psh property) of K\"ahler potentials (\cf Remark \ref{globalpositivityKolodziej}). Thus when we graft the real MA solution from $Sk(X)$ to $X_t$, we must guarantee the global positivity. The difficulty lies in the non-generic regions where the complex structure on $X_t$ is highly singular.


\end{itemize}

These problems point naturally towards NA geometry:
\begin{itemize}
\item The NA MA-real MA comparison property is a natural way to produce solutions.

\item The hybrid topology is a natural topology to compare $X_t$ with $X_K^{an}$, which contains the essential skeleton.

\item The notion of semipositive metric is built into NA geometry.
\end{itemize}

\begin{rmk}
	A byproduct of the non-archimedean approach, is that the limit of Calabi-Yau local potentials is in fact independent of subsequence, since the non-archimedean analogue of the Calabi-Yau metric is known to be unique.
\end{rmk}

\subsubsection{Grafting the real MA solution}

Let $(\mathcal{X}, \mathcal{L})$ be a semistable snc model with $\mathcal{L}|_X=L$. The NA pluripotential theory provides a continuous semipositive metric $\norm{\cdot}_{CY}=\norm{\cdot}_{\mathcal{L} }e^{-\phi_0}$ on $L$ over $X_K^{an}$ solving the NA MA equation (\ref{NAMACY}), which we assume henceforth satisfies the NA MA-real MA comparison property, so $\phi_0$ solves the real MA equation over the $n$-dimensional open faces $\text{Int}(\Delta_J)$ of the essential skeleton $Sk(\mathcal{X} )$ (\cf section \ref{NACalabi}).

\begin{prop}\label{regularisationlemma}\cite[Lemma 4.1, 4.2]{LiNA}
Given any $\epsilon\ll 1$, and let $t$ be small enough depending on $\epsilon$. There is a K\"ahler metric $\omega_{\psi,t}$, such that 
\begin{itemize}
\item On $\text{Log}_{ \mathcal{X} }^{-1}(\text{Int}(\Delta_J))$, the local K\"ahler potentials 
$\phi_{J,t}$ of $\omega_{\psi,t}$ can be chosen to  satisfy $|\phi_{J,t}-  \phi_0\circ \text{Log}_{ \mathcal{X}}|<\epsilon$.

\item 
The total variation $ \int_{X_t} | \frac{ |\log |t||^n \omega_{\psi,t}^n }{ (L^n) }  -d\mu_t|  <\epsilon.$

\item The K\"ahler potential of $\omega_{\psi,t}$ relative to a fixed  Fubini-Study reference metric, is uniformly bounded independent of $t, \epsilon$.  
\end{itemize}
\end{prop}

\begin{proof}
(Sketch)
\begin{itemize}
	\item We first $C^0$ approximate the NA metric $\norm{\cdot}_{CY}$ by some NA \textbf{Fubini-Study metric}, which arises naturally as a hybrid topology limit of usual Fubini-Study metrics on $X_t$ (\cf section \ref{FubiniStudyapproximation}). The Fubini-Study metrics are positive, and by construction their local potentials differ from $ \phi_0\circ \text{Log}_{ \mathcal{X}}$ by an arbitrarily small amount in the $C^0$ sense.

	\item We do not have direct control on the volume form of the Fubini-Study metrics; the degrees of the associated projective embeddings are gigantic. In contrast, the volume form of the local potential $\phi_0\circ \text{Log}_{ \mathcal{X}}$ has negligible difference from $\frac{(L^n)}{|\log |t||^n}d\mu_t$, by the volume asymptote in section \ref{volumeasymptoteessentialskeleton} and the real MA equation (\ref{realMACY}).

	\item  The idea is to perform a further \textbf{regularization}. We modify the Fubini-Study metric in the generic region of $X_t$, so that it essentially agrees with $\phi_0\circ \text{Log}_{ \mathcal{X}}$ in the generic region up to $C^2$-small error. In this step we appealed also to the regularity theory of real MA equation. The end result is $\omega_{\psi,t}$, which is K\"ahler by construction.

	\item   In the non-generic region, we do not perform regularization. Since the generic region already takes up $99.9\%$ of the $\omega_{\psi,t}^n$ measure for $|t|\ll 1$, the non-generic region has negligible total measure. We use this to argue for the total variation bound.	
\end{itemize}
\end{proof}

\subsubsection{$C^0$-convergence of the potential}

It remains to show
\begin{prop}\label{C0convergence}
Up to slightly shrinking the domains, the Calabi-Yau metrics on  $U_{J,t}$ admit local potential functions $\phi_{CY, J,t}$ such that $\omega_{CY,t}=dd^c \phi_{CY, J,t}$, and $\norm{\phi_{CY,J,t}- \phi_0\circ \text{Log}_{\mathcal{X}}}_{C^0} \to 0$ as $t\to 0$. 
\end{prop}

Prop. \ref{regularisationlemma} says that  the local potential of $\omega_{\psi,t}$ and $\phi_0\circ \text{Log}_{ \mathcal{X}}$ differ negligibly in the $t\to 0$ limit in the $C^0$-sense. Ideally, one would like to use some version of $L^1$-volume stability to conclude the $C^0$-smallness of the relative potential between $\omega_{\psi,t}$ and $\omega_{CY,t}$. Unfortunately, due to the difficulty of regularization in the non-generic region, there is very little control on the volume density of $\omega_{\psi,t}$ in the non-generic region, and we cannot conclude a Skoda type estimate like (\ref{Skodaassumption}) for $\omega_{\psi,t}$. This technical problem causes an \emph{asymmetry} between $\omega_{\psi,t}$ and $\omega_{CY,t}$, and only `one half' of the $L^1$-volume stability estimate (\cf section \ref{Toolsfrompsh}) applies, which is why we designed Theorem \ref{UniformL1stabilitythm}.

After the dust settles, Theorem \ref{UniformL1stabilitythm} implies that $\phi_{CY,J,t}- \phi_0\circ \text{Log}_{\mathcal{X}}$ \emph{concentrates near its minimum value (normalized to be zero) on a subset with almost $100\%$ of the Calabi-Yau measure} (\cf \cite[Prop 4.4, Cor. 4.6]{LiNA}). More precisely, for any given small number $\kappa,\lambda\ll 1$, then for sufficiently small $t$, the measure
\begin{equation}
d\mu_t(   \phi_{CY,J,t}- \phi_0\circ \text{Log}_{\mathcal{X}} \geq \kappa/4  )<\lambda.
\end{equation}
On a slightly shrinked version of $U_{J,t}$, this can be improved to the $C^0$-control
\[
0\leq \phi_{CY,J,t}- \phi_0\circ \text{Log}_{\mathcal{X}} <\kappa,
\]
by a slightly tricky application of the mean value inequality (\cf \cite[Thm. 4.7]{LiNA}).
Since $\kappa$ is arbitrary, this achieves  Prop. \ref{C0convergence}, which verifies the hypothesis of Prop. \ref{C0convergence}, whence the weak metric version of the SYZ conjecture.

\begin{rmk}
The $C^0$ convergence statement only applies to the generic region. We do not know the answer to

\begin{Question}
Do the potentials of the CY metrics on $X_t$ converge to the NA CY metric $\norm{\cdot}_{CY}$ on $X_K^{an}$ globally in the hybrid topology?
\end{Question}
\end{rmk}

\subsection{Strategy II: a priori limit}\label{StrategyII}

The second strategy does not appeal to NA geometry, and is independent of section \ref{StrategyI}.

\begin{itemize}
\item  Argue a priori that the local potential functions of the Calabi-Yau metrics on $U_{J,t}\subset X_t$ converge subsequentially to some convex function on the open $n$-dimensional faces of $Sk(X)$, in the $C^0$-norm.

\item Argue that the convex function satisfies the real MA equation.

\end{itemize}

The strategy is general, except for a delicate problem which we only solved in the very special case for the Fermat family (\cf Theorem \ref{Fermatthm} \cite{LiFermat}).

\subsubsection{Producing convex functions}

Recall the logarithm map $\text{Log}_t: (\C^*)^n\to \R^n$.
\[
\text{Log}_t (z_1,\ldots z_n)=\frac{1}{\log |t|} (\log |z_1|,\ldots \log |z_n|).
\]
Consider an open convex subset $U\subset \R^n$, and let $\phi$ be a psh function on $\text{Log}_t^{-1}(U)\subset (\C^*)^n$.

\begin{lem}\label{pshconvexity}\cite[Lemma 4.3]{LiFermat}
The fibrewise $T^n$ average function \[
	\bar{\phi}(x_1,\ldots x_n)= \frac{1}{(2\pi)^n} \int_{T^n} \phi( e^{x_1\log |t|+ i\theta_1}, \ldots e^{x_n\log |t|+ i\theta_n}) d\theta_1\ldots d\theta_n
	\]
	is a convex function in the variables $x_1,\ldots x_n$. 
\end{lem}

\begin{proof}
Since the function $\bar{\phi}$ is an average of psh functions, it is psh as a $T^n$-invariant function on $\text{Log}_t^{-1}(U)$. Such functions correspond to convex functions downstairs.
\end{proof}

An important intuition is that \emph{on sufficiently collapsed toric regions inside $(\C^*)^n$, bounded K\"ahler potentials have a strong tendency to be approximated by convex functions}.

\begin{prop}
Assume $\norm{\phi}_{C^0}$ has a uniform bound independent of $t$. Then after shrinking $U$ by a small amount independent of $t$, we have
\begin{itemize}
\item The convex function $\bar{\phi}$ has a Lipschitz bound $|\bar{\phi}(x)-\bar{\phi}(x')|\leq C|x-x'|.$

\item There is an upper bound $\phi-\bar{\phi}\leq  \frac{C}{ |\log |t||^{1/2} }$.

\item  On each logarithmic dyadic scale $U_{a}=\{ a_i\leq \log |z_i| \leq 2a_i, \forall i\}\subset U$, the $L^1$-integral
\[
\int_{U_{a}} |\phi-\bar{\phi}|\prod \sqrt{-1}d\log z_i\wedge d\overline{\log z_i} \leq \frac{C}{ |\log |t||^{1/2} }.
\]

\item  There is an \emph{improved Skoda inequality} with uniform constants $\alpha,C$ independent of $t$:
\[
\int_{U} e^{-\alpha |\log |t||^{1/2} (\phi-\bar{\phi})   } d\mu_t \leq C.
\]

\end{itemize}

\end{prop}

\begin{proof}
(Sketch)
\begin{itemize}
\item
Bounded convex functions automatically have Lipschitz bound on slightly shrinked convex domains. 

\item The second item follows from a slightly tricky application of mean value inequality for subharmonic functions, \cf \cite[section 4.3]{LiFermat}.

\item The third item is because the function $\phi-\bar{\phi}$ has mean value zero, so an upper bound implies an $L^1$-bound, \cf  \cite[section 4.3]{LiFermat}.

\item  One first apply the basic Skoda estimate Thm. \ref{Skodabasicversion} to the function $\phi$ on each logarithmic dyadic scale, where $\bar{\phi}$ is almost constant by the Lipschitz bound. Then we sum over all the logarithmic dyadic scales (\cf \cite[section 4.6]{LiFermat}).

\end{itemize}
\end{proof}

\begin{rmk}\label{improvedSkoda}
The improved Skoda estimate is one of the main discoveries in \cite{LiFermat}. Intuitively, this means $\phi-\bar{\phi}$ can only be significantly \emph{below} $- \frac{Const}{|\log |t||^{1/2}}$ on sets with exponentially small measure. Compounded with the upper bound $\phi-\bar{\phi}\leq  \frac{C}{ |\log |t||^{1/2} }$, this means for sufficiently small $t$, an arbitrary bounded psh function $\phi$ is very close to the convex function $\bar{\phi}$ except on exponentially small measure.	
\end{rmk}

In our applications, the psh functions $\phi$ arise from the local potentials $\phi_{CY,J, t}$ of the Calabi-Yau metrics $\omega_{CY,t}$ on toric charts inside $X_t$. Since $\omega_{CY,t}$ has uniformly bounded potential with respect to Fubini-Study reference metrics (\cf Thm. \ref{UniformLinfty}), it is easy to arrange the local potentials $\phi_{CY,J,t}$ on toric charts to be uniformly bounded, whence the convex functions $\bar{\phi}_{CY,J,t}$ are also uniformly bounded. Using the Lipschitz bound, by Arzela-Ascoli, we can extract a collection of subsequential limits as $t\to 0$. By construction $\bar{\phi}_{CY,J,t}\to \bar{\phi}_{J,0}$ in the $C^0_{loc}$ sense on the interior of the $n$-dimensional faces of $Sk(X)$.

\subsubsection{$C^0$-convergence of the potential and extension problem}

We aim to show $\norm{\phi_{CY,J,t}- \bar{\phi}_{J,0}\circ \text{Log}_\mathcal{X}}_{C^0}$ on the slightly shrinked $U_{J,t}$ converges to zero along the subsequence.
We know $\norm{\bar{\phi}_{CY,t} -\bar{\phi}_{J,0} }_{C^0}\to 0$ along the subsequence, and from Remark \ref{improvedSkoda}, we know $|\phi_{CY,J,t}- \bar{\phi}_{CY,J,t}\circ \text{Log}_\mathcal{X}|$ is small \emph{except on a subset with small measure}. 
Removing this small measure problem, is however rather subtle, and requires a global argument.

The strategy is:
\begin{itemize}
\item (`\textbf{Extension problem}') Find a \emph{global K\"ahler metric} $\omega_{\psi,t}$ on $X_t$ whose local potentials on $U_{J,t}$ agree with $\bar{\phi}_{J,0}\circ \text{Log}_{\mathcal{X}}$ up to $C^0$ small error, and whose potential with respect to a fixed Fubini-Study metric is bounded independent of $t$.

\item  (\textbf{Potential stability estimate}) We can then consider the potential $\phi_{CY,rel}$ of the Calabi-Yau metric $\omega_{CY,t}$ relative to $\omega_{\psi,t}$. A small upper bound for $\phi_{CY,rel}$ on  $U_{J,t}$ follows from $\phi_{CY,J,t}- \bar{\phi}_{CY,J,t}\circ \text{Log}_\mathcal{X} \leq \frac{C}{|\log |t||^{1/2}}$. We also know a small lower bound on the $\phi_{CY,rel}$ holds except on a set with very small measure, and then an application of Theorem \ref{UniformL1stabilitythm} concludes a small lower bound on $\inf \phi_{CY,rel}$. We emphasize that the \emph{global positivity} of K\"ahler metrics is essential for this argument.

The net conclusion is that $\phi_{CY,rel}$ is $C^0$-small on a slightly shrinked version of $U_{J,t}$. This amounts to the smallness of $\norm{\phi_{CY,J,t}- \bar{\phi}_{J,0}\circ \text{Log}_\mathcal{X}}_{C^0}$, which is our goal.

\end{itemize}

The extension problem is about patching local potentials to global K\"ahler potentials, and the difficulty is to achieve psh property in the non-generic region. The core problem, which is not satisfactorily solved in general, is 

\begin{Question}
Can we sufficiently explicitly characterize the class of convex potentials on $Sk(X)$ that can be regarded as limits of K\"ahler potentials on $X_t$?
\end{Question}

The extension problem is solved in an ad hoc way for the Fermat family, and constitutes the most technical part of \cite{LiFermat}.\footnote{Technically, the paper \cite{LiFermat} does not use the language of dual complexes and essential skeletons, but proceed via explicit charts controlled by tropical geometry.} Recall the Fermat family embeds into an ambient projective space $\mathbb{CP}^{n+1}$. Our strategy is to produce the extension $\omega_{\psi,t}$ as a \emph{toric K\"ahler metric} on $\mathbb{CP}^{n+1}$, and then restrict to $X_t$, which guarantees the global positivity. Ensuring that $\omega_{\psi,t}$ agrees with the local convex functions up to $C^0$-small error is a delicate matter, that involves the explicit tropical hypersurface combinatorics, exploits the large amount of discrete symmetry of the Fermat family, and uses a double Legendre transform construction \cite{LiFermat}.

\begin{rmk}
The motivation for toric K\"ahler metrics on $\mathbb{CP}^{n+1}$ is as follows. The toric property is a natural way to reduce general K\"ahler potentials to convex functions. The idea of extension to an ambient space, is based on 
\begin{prop}\label{extensionKahlercurrent}
	(\cite[Thm. B]{Coman}) Let $(X,\omega)$ be a projective manifold with a K\"ahler form representing an integral class, and $Y$ be a smooth subvariety of $X$. Then any $\phi \in PSH(Y,\omega|_Y)$ extends to $\phi\in PSH(X,\omega)$.
\end{prop}

\end{rmk}

\subsubsection{Real MA metric}

To complete the circle, we need

\begin{lem}
The limiting convex potentials $\bar{\phi}_{J,0}$ solve the \emph{real MA equation} (\ref{realMACY}) on the interior of $\Delta_J$.
\end{lem}

The strategy is to pass the complex MA equation on $U_{J,t}$ to the limit. This is feasible, morally because the complex MA operator is weakly continuous under the $C^0$-convergence of potentials. In our setting, an extra subtlety is that the sequence of potentials are defined on different manifolds, and the modification of the usual arguments are carried out in \cite[section 5.1]{LiFermat}.

\subsubsection{Relation to NA geometry}

The a priori limit strategy does not explicitly appeal to NA geometry. Its principal remaining difficulty is the extension problem. Based on the experience with the Fermat example, we anticipate that extension to toric metrics on ambient toric varieties is a useful technique, and the problem may have a substantially combinatorial aspect. As we emphasized in section \ref{Comparisonproperty}, an explicit class of convex potentials would also be essential for a direct formulation of the real MA equation, which is likely needed for more refined questions such as the affine structure and the singular set of the real MA metric on $Sk(X)$ (\cf the Kontsevich-Soibelman conjecture in section \ref{KontsevichSoibelmanconj}).

In contrast, the NA pluripotential theory is built around the central concept of NA semipositive metrics on $X_K^{an}$, which extend up to $C^0$-small error to K\"ahler potentials on $X_t$ via the Fubini-Study approximation. In that respect, NA pluripotential theory may be viewed as a disguised solution of the extension problem. To make contact with differential geometric applications, however, requires some additional hypothesis such as the NA MA-real MA comparison property. Comparing the difficulties in the two strategies, we speculate that proving the NA MA-real MA comparison property requires a more concrete characterization of NA semipositive metrics, perhaps of explicitly combinatorial nature.

\end{document}